\documentclass[11pt, letterpaper]{amsart}
\usepackage{amscd, amssymb}
\usepackage{enumerate}
\usepackage[all]{xypic}
\usepackage{multirow}
\usepackage{hhline}
\usepackage{verbatim}
\usepackage{mathdots}
\usepackage{comment}
\usepackage{tikz}
\usepackage{appendix}
\usepackage{graphicx}
\usepackage{mathtools}
\usepackage{extpfeil}
\usepackage{accents}

\usepackage{dynkin-diagrams}

\usetikzlibrary{decorations.markings}

\theoremstyle{plain}
\newtheorem{Thm}{Theorem}[section]
\newtheorem*{Thm*}{Theorem}

\newtheorem{Prop}[Thm]{Proposition}
\newtheorem{Lem}[Thm]{Lemma}
\newtheorem{Conj}[Thm]{Conjecture}

\theoremstyle{definition}
\newtheorem{Def}[Thm]{Definition}

\newtheorem{Ex}[Thm]{Example}

\theoremstyle{remark}
\newtheorem{Rmk}[Thm]{Remark}

\numberwithin{equation}{section}

\newcounter{myenum}
  {\end{list}}

\newcounter{myenumsec}
  {\end{list}}

\newcommand{\C}{\mathbb C}

\newcommand{\Q}{\mathbb{Q}}

\newcommand{\Z}{\mathbb{Z}}

\newcommand{\one}{\mathbf{1}}

\newcommand{\Gm}{\mathbb{G}_m}
\newcommand{\Ga}{\mathbb{G}_a}
\newcommand{\GL}{\operatorname{GL}}
\newcommand{\SL}{\operatorname{SL}}
\newcommand{\PGL}{\operatorname{PGL}}

\newcommand{\Sp}{\operatorname{Sp}}
\newcommand{\OO}{\operatorname{O}}
\newcommand{\SO}{\operatorname{SO}}

\newcommand{\la}{\langle}
\newcommand{\ra}{\rangle}

\newcommand{\Ind}{\operatorname{Ind}}
\newcommand{\Hom}{\operatorname{Hom}}
\newcommand{\Span}{\operatorname{span}}

\newcommand{\Int}{\operatorname{Int}}
\newcommand{\diag}{\operatorname{diag}}
\newcommand{\Irr}{\operatorname{Irr}}

\newcommand{\Lcal}{\mathfrak{L}}

\newcommand{\oalpha}{{\overline{\alpha}}}

\newcommand{\qand}{{\quad\text{and}\quad}}
\newcommand{\st}{\,:\,}

\pgfkeys{/Dynkin diagram, edge-length=1cm, root-radius=.1cm}

\title[On dual groups of symmetric varieties]{On dual groups of symmetric varieties and\\
  distinguished representations of $p$-adic groups}

\author{Shuichiro Takeda}
\address{Department of Mathematics\\
Graduate School of Science\\
Osaka University\\
Toyonaka, Osaka 560-0043\\
Japan}
\email{takedas@math.sci.osaka-u.ac.jp}

\begin{document}

\maketitle

\begin{abstract}
Let $X=H\backslash G$ be a spherical variety over a $p$-adic field. Assume $G$ is split. Let $\widehat{G}$ be the Langlands dual group of $G$.  There is a complex group $\widehat{G}_X$ whose root datum is the little Weyl group of $X$. It was proposed by Sakellaridis-Venkatesh and fully proven by Knop and Schalke that there is a homomorphism $\widehat{\varphi}_X:\widehat{G}_X\times\SL_2(\C)\to \widehat{G}$. Conjecturally, this detects the $H$-distinguished representations of $G$.

In this strictly utilitarian note, assuming $X$ is a symmetric variety, we give a more conceptual way of constructing the homomorphism $\widehat{\varphi}_X:\widehat{G}_X\times\SL_2(\C)\to \widehat{G}$, and make a few conjectures on how $\widehat{\varphi}_X$ is related to $H$-distinguished representations of $G$ by using various known examples and conjectures, especially in the framework of the theory of Kato-Takano and Lagier on relative cuspidality and relative square integrability. We will also show that the local Langlands parameter of the trivial representation of $G$ factors through $\widehat{\varphi}_X$ for any symmetric variety $X=H\backslash G$.

\end{abstract}


\section{Introduction}


Let $F$ be a non-archimedean local field of characteristic 0, and $G$ a connected reductive group split over $F$ equipped with an $F$-involution
$\theta$. Let $H$ be the subgroup of $\theta$-fixed points of $G$, so that the quotient $X:=H\backslash G$ is a symmetric variety.

An irreducible admissible representation $(\pi, V)$ of $G$ is said to be $H$-distinguished if there exists a nonzero $H$-invariant linear form
$\lambda: V\rightarrow\C$, so that $\lambda(\pi(h)v)=\lambda(v)$ for all $h\in H$ and $v\in V$, namely $\Hom_H(\pi,\one)\neq 0$. It is then expected that the (conjectural) local Langlands parameter $\varphi_\pi:WD_F\to \widehat{G}$ of $\pi$ should factor through some smaller subgroup of $\widehat{G}$. (Here, as usual, $WD_F$ is the Weil-Deligne group of $F$ and $\widehat{G}$ is the Langlands dual group of $G$.) Indeed, in their monumental work \cite{SV}, Sakellaridis and Venkatesh, motivated by the work of Gaitsgory and Nadler \cite{Gaitsgory-Nadler}, proposed a ``dual group" $\widehat{G}_X$ of $X$ and a (conjectural) homomorphism $\widehat{G}_X\times\SL_2(\C)\to \widehat{G}$ in the broader context of spherical varieties, so that the local Langlands parameter $\varphi_\pi$ should factor through the map $\widehat{G}_X\times\SL_2(\C)\longrightarrow \widehat{G}$. The construction of the map was later completed by Knop and Schalke \cite{Knop_Schalke}. Their approach is of a highly combinatorial nature, relying on the classification of spherical varieties. Though their method offers a precise and systematic framework, it can be technically demanding for those less acquainted with the subject, including the author of the present paper.

In this paper, we construct the map in the case where $X=H\backslash G$ is a symmetric variety by using the involution $\theta$ and the root datum of $G$ as the only input data. To be more specific, the given involution $\theta$ naturally gives rise to an involution on the root datum of $G$ and hence of $\widehat{G}$. This naturally divides the root datum into two root data: the $\theta$-invariant root datum and the $\theta$-split root datum. We denote their corresponding complex groups by $\widehat{G}^+$ and $\widehat{G}^-$, respectively. To be more explicit, for a suitably chosen torus $T$, we define
\[
\widehat{T}^+=\{t\in\widehat{G}\st \theta(t)=t\}^\circ\qand \widehat{T}^-=\{t\in\widehat{T}\st \theta(t)=t^{-1}\}^\circ,
\]
where $\widehat{T}$ is the dual torus of $T$. Then the inclusions
\[
\widehat{T}^+\subseteq\widehat{T}\qand \widehat{T}^-\subseteq \widehat{G}
\]
extend to inclusions
\[
\widehat{G}^+\subseteq \widehat{G}\qand \widehat{G}^-\subseteq \widehat{G}.
\]
We set
\[
\widehat{G}_X:=\widehat{G}^-
\]
and let
\[
\widehat{\varphi}^+:\SL_2(\C)\longrightarrow \widehat{G}^+
\]
be the so-called principal $\SL_2$-homomorphism of $\widehat{G}^+$. We will then show that the two subgroups $\widehat{\varphi}^+(\SL_2(\C))$ and $\widehat{G}^-$ commute with each other, so that we have a homomorphism
\[
\widehat{\varphi}_X:\widehat{G}_X\times\SL_2(\C)\longrightarrow \widehat{G}.
\]

We will also recast the theory of $H$-matrix coefficients developed by Kato and Takano \cite{KT-Cuspidal, KT-Square} and Lagier \cite{Lagier}, which was later augmented by the author \cite{Takeda}.

Let us recall the theory here. Let $(\pi, V)$ be an $H$-distinguished representation with unitary central character, and $\lambda\in \Hom_H(\pi,\one)$ nonzero. For each $v\in V$, define
\[
\varphi_{\lambda, v}:H\backslash G\longrightarrow \C,\quad \varphi_{\lambda, v}(g)=\la\lambda, \pi(g)v\ra,
\]
for $g\in H\backslash G$. We call $\varphi_{\lambda, v}$ an $H$-matrix coefficient (with respect to $\lambda$). We say
\begin{enumerate}[(a)]
\item $(\pi, V)$ is relatively cuspidal if $\varphi_{\lambda, v}\in C_c^\infty(Z_GH\backslash G)$ for all $\lambda\in\Hom_H(\pi,\one)$ and all $v\in V$;
\item $(\pi, V)$ is relatively square integrable if $\varphi_{\lambda, v}\in L^2(Z_GH\backslash G)$ for all $\lambda\in\Hom_H(\pi,\one)$ and all $v\in V$;
\item $(\pi, V)$ is relatively tempered if for all $\epsilon>0$ we have $\varphi_{\lambda, v}\in L^{2+\epsilon}(Z_GH\backslash G)$ for all $\lambda\in\Hom_H(\pi,\one)$ and all $v\in V$.
\end{enumerate}
With the assumption that the residue characteristic of $F$ is odd, in \cite{KT-Cuspidal} Kato and Takano have shown that relative cuspidality is detected by the vanishing of the ``Jacquet-module" of $\lambda$, and in \cite{KT-Square} they have established the Casselman type criterion for relative square integrability in terms of exponents. Further, in \cite{Takeda}, the author has established the Casselman type criterion for relative temperedness.

We then make the following conjecture.
\begin{Conj}\label{C:main_conjecture}
Let $\pi$ be an $H$-distinguished irreducible admissible representation of $G$, so that $\Hom_H(\pi, \one)\neq 0$.
\begin{enumerate}[(I)]
\item The (conjectural) local Langlands parameter $\varphi_\pi:WD_F\to \widehat{G}$ of $\pi$ factors through
\[
\varphi_\pi:WD_F\longrightarrow  \widehat{G}_X\times\SL_2(\C)\xrightarrow{\;\widehat{\varphi}_X\;} \widehat{G},
\]
where the map $WD_F\to\SL_2(\C)$ is given by
\[
w\mapsto\begin{pmatrix}|w|^{\frac{1}{2}}&\\&|w|^{-\frac{1}{2}}\end{pmatrix}
\]
for all $w\in WD_F$.
\item Assume (I) holds.
\begin{enumerate}[(a)]
\item If the image of $WD_F\to \widehat{G}_X$ is not in a proper Levi of $\widehat{G}_X$ and the $\SL_2(\C)$-factor of $WD_F$ maps trivially, then $\pi$ is relatively cuspidal.
\item $\pi$ is relatively square integrable if and only if the image of $WD_F\to \widehat{G}_X$ is not in a proper Levi of $\widehat{G}_X$.
\item $\pi$ is relatively tempered if and only if the image of $WD_F\to \widehat{G}_X$ is bounded modulo center.
\end{enumerate}
\end{enumerate}
\end{Conj}

We will see that the converse of Conjecture (I) does not always hold; namely even if $\pi$ has an $L$-parameter $\varphi:WD_F\to \widehat{G}$ factoring through $\widehat{\varphi}_X$, it could be the case that $\pi$ is not $H$-distinguished. However, such examples seem to be quite degenerate and presumably the converse of Conjecture (I) almost always holds. Also the converse of Conjecture (II-a) fails already in the so-called group case due to the existence of an $L$-packet which contains a supercuspidal representation and a non-supercuspidal discrete series representation at the same time. But this seems to be the only obstruction to the converse of Conjecture (II-a), and presumably it can be modified by considering all the members of the $L$-packet, although at this moment the author does not know enough examples to make any precise suggestion to modify it. At any rate, we verify that our conjectures are consistent with numerous known examples and conjectures.

\quad

Note that the trivial representation $\one$ of $G$ is $H$-distinguished for any $H\subseteq G$. Hence by Conjecture (I) the $L$-parameter of $\one$ must factor through $\widehat{\varphi}_X$ for any symmetric variety $X=H\backslash G$. We prove this assertion in Theorem \ref{T:trivial_parameter}.

\quad

The following is the overall structure of the paper. In the next section (Section 2), we review the method called the folding of a root datum by following \cite{Knop_Schalke}. In Section 3, after reviewing some generalities on involutions on root data, we construct the group $G^+$ whose root datum is the $\theta$-invariant root datum. In Section 4, we discuss the notion of the Dynkin diagram of an involution $\theta$ and the diagram of each $\theta$-non-invariant simple root, and illustrate how this theory works by using some of the important examples. In Section 5, we construct the group $G^-$ whose root datum is (the reduced form of) the $\theta$-split root datum. In Section 6, we prove the commutativity of $G^-$ and the image of the principal $\SL_2$ homomorphism of $G^+$. In Section 7, we define our dual group $\widehat{G}_X$ and the homomorphism $\widehat{\varphi}_X:\widehat{G}\times\SL_2(\C)\to\widehat{G}$. In Section 8, we prove that the $L$-parameter of the trivial representation $\varphi_{\one}$ always factors through $\widehat{\varphi}_X$. In Section 9, we discuss some instances of the distinguished representations and show that they are consistent with our Conjectures. Finally, in Appendix, we give details of rank one involutions.

\quad

\begin{center}{\bf Notation and assumptions}\end{center}

We let $F$ be a nonarchimedean local field of characteristic 0, and $WD_F=W_F\times \SL_2(\C)$ its Weil-Deligne group. For each $w\in WD_F$, we denote its norm by $|w|$. For a (split) reductive group $G$ over $F$, we let $\widehat{G}$ be the Langlands dual group of $G$. We usually identify $G$ with its $F$-points $G(F)$, and denote the center of $G$ by $Z_G$. By an $F$-involution $\theta$ on $G$, we mean a group homomorphism $\theta:G\to G$ defined over $F$ such that $\theta^2=1$, and we denote its fixed points by $H$. For each $\varepsilon\in G$, we let $\Int(\varepsilon)$ be the automorphism on $G$ defined by $g\mapsto \varepsilon g\varepsilon^{-1}$. Note that $\Int(\varepsilon)$ is an involution if and only if $\varepsilon^2\in Z_G$.

We let $\Irr(G)$ be the set of equivalence classes of irreducible admissible representations of $G$, and denote by $\one$ the trivial representation of $G$. A representation $\pi\in\Irr(G)$ is said to be $H$-distinguished if $\Hom_H(\pi,\one)\neq 0$, and call each nonzero $\lambda\in\Hom_H(\pi,\one)$ an $H$-period. For each $\pi\in\Irr(G)$ we denote its contragredient by $\check{\pi}$.

Let $G$ be a split reductive group and $\Phi$ the set of roots with respect to a maximal torus $T\subseteq G$. For each $\alpha\in\Phi$, there is an isomorphism $u_\alpha:\Ga\to U_\alpha$ onto a unique closed subgroup $U_\alpha\subseteq G$ called the root subgroup of $\alpha$ such that $t u_\alpha(x) t^{-1}=u_\alpha(\alpha(t)x)$ for all $t\in T$ and $x\in\Ga$. The group $G$ is generated by $T$ and all the $U_\alpha$'s.

For the notation and convention of root systems, we follow the book by Knapp \cite{Knapp} unless otherwise stated, though it should be warned that in \cite{Knapp} the set of all roots is denoted by $\Delta$ whereas we denote by $\Delta$ the set of simple roots. See \cite[p.150]{Knapp} for classical groups and \cite[p.180-184]{Knapp} for exceptional groups. For example, for the root system of type $A_{n-1}$, we use $\{e_1,\dots,e_n\}$ and $\{e_1^\vee,\dots,e_n^\vee\}$ for the standard basis of the character lattice and the cocharacter lattice, respectively, and write $\alpha_i=e_i-e_{i+1}$ and $\alpha_i^\vee=e_i^\vee-e_{i+1}^\vee$ for the simple roots and coroots, respectively. Similarly for other types of root systems.

\quad

\begin{center}{\bf Acknowledgements}\end{center}

The author would like to thank Yiannis Sakellaridis, Chuijia Wang and Jiandi Zou for their comments on an earlier version of the paper. A part of the paper was completed when the author was visiting the National University of Singapore in the summer of 2023. He would like to thank their hospitality. The author was partially supported by the Sumitomo Foundation Fiscal 2024 Grant for Basic Science Research Projects J230603025, and  JSPS KAKENHI grant number 24K06648.
\quad


\section{Folding of a root datum}

In this section, we let $G$ be a split reductive group over an arbitrary field $k$ of $\operatorname{char} k\neq 2$. Here $k$ does not have to be our nonarchimedean local field $F$. Indeed, the case we have in mind is when $k=\C$ and $G$ is the Langlands dual group of a reductive group over $F$.

The main goal of this section is to describe the method called ``folding", which will be used to produce the two subgroups $\widehat{G}^+$ and $\widehat{G}^-$ discussed in the Introduction. 

\subsection{Additively closed root sub-datum}
Let $\Phi=(X, R, X^\vee, R^\vee)$ be a root datum, where $R$ and $R^\vee$ are the sets of roots and coroots, respectively. A quadruple $\Psi=(X, S, X^\vee, S^\vee)$ is said to be a root sub-datum of $\Phi$ if $\Psi$ is a root datum with $S\subseteq R$ and $S^\vee\subseteq R^\vee$. Further, $\Psi$ is said to be additively closed if $S=\Z S\cap R$; in other words if $\alpha, \beta\in S$ are such that $\alpha+\beta\in R$, then $\alpha+\beta\in S$.

\begin{Prop}\label{P:subgroup_generated_by_sub-datum}
Assume $\Phi$ is a root datum of a split reductive group $G$. Then a root sub-datum $\Psi$ of $\Phi$ generates a subgroup in $G$ whose root datum is $\Psi$ if and only if $\Psi$ is additively closed, in which case the subgroup is generated by $T$ and the root subgroups $U_\alpha$ for all $\alpha\in S$.
\end{Prop}

This is stated in the proof of \cite[Theorem 7.3]{Knop_Schalke} with the citation \cite{BDS}. Yet, this can be easily proven by using the well-known relations for $T$ and $U_\alpha$'s of the reductive group $G$. But, for $\alpha, \beta\in S$ such that $\alpha\neq \pm\beta$, there is a commutator relation
\begin{equation}\label{E:commutator_relation}
u_\alpha(x)u_\beta(y)u_\alpha(x)^{-1}u_\beta(y)^{-1}=\prod_{\substack{i\alpha+j\beta\in R\\ i, j>0}}u_{i\alpha+j\beta}(c_{\alpha, \beta, i, j}x^iy^j),
\qquad(x, y\in k),
\end{equation}
where $c_{\alpha, \beta, i, j}\in k$ is known as the structure constant. This is why one needs $\Psi$ to be additively closed.

For each subset $\Sigma\subseteq R$, we set
\[
\Sigma^{ac}=\Z\Sigma\cap R=\{\sum_{n_\alpha\in\Z} n_\alpha\alpha\st \alpha\in \Sigma\}\cap R,
\]
and call it the {\it additive closure} of $\Sigma$.
\begin{Lem}\label{L:additive_closure_is_root_datum}
With the above notation, the quadruple $(X, \Sigma^{ac}, X^\vee, \Sigma^{ac\vee})$ is an additively closed root sub-datum, where $\Sigma^{ac\vee}$ is the set of the coroots corresponding to the roots in $\Sigma^{ac}$.
\end{Lem}
\begin{proof}
By definition, $\Sigma^{ac}$ is additively closed. Hence it remains to show that it is indeed a root datum. First we show that $\Sigma^{ac}$ is closed under the root reflections. But each $\alpha\in\Sigma^{ac}$ is written as $\alpha=\sum n_i\alpha_i$ for $n_i\in\Z$ and $\alpha_i\in\Sigma$, and hence for each $\beta\in\Sigma^{ac}$ we have
\[
s_{\alpha}(\beta)=\beta-\la \beta, \alpha^\vee\ra\alpha=\beta-\sum n_i\la \beta, \alpha^\vee\ra\alpha_i\in \Z\Sigma\cap R=\Sigma^{ac}.
\]
To see $\Sigma^{ac\vee}$ is closed under the coroot reflections, use
\[
s_{\alpha^\vee}(\beta^\vee)=s_{\alpha}(\beta)^\vee
\]
for all coroots $\alpha^\vee, \beta^\vee\in R$. (See \cite[Lemma 3.2.4]{Conrad-Gabber-Prasad}.)
\end{proof}

We also call the root datum $(X, \Sigma^{ac}, X^\vee, \Sigma^{ac\vee})$ the {\it additive closure} of $\Sigma$.

\quad

Even if a root sub-datum $\Psi=(X, S, X^\vee, S^\vee)$ is additively closed, its dual $\Psi^\vee=(X^\vee, S^\vee, X, S)$ does not have to be additively closed. For example, if $G=\Sp_4$ and $S$ is the set of all long roots, then $S$ is additively closed but $S^\vee$, which is the set of all short coroots, is not additively closed. Indeed, in this case, $(X, S, X^\vee, S^\vee)$ is the root datum of $\SL_2\times\SL_2$ and we have the natural embedding $\SL_2\times\SL_2\to\Sp_4$. But the dual $(X^\vee, S^\vee, X, S)$ is the root datum of $\PGL_2\times\PGL_2$ and there is no embedding $\PGL_2\times\PGL_2\to\Sp_4^\vee=\SO_5$.

Yet, let us mention the following, though we do not use it for this paper.
\begin{Lem}
Let $\Psi$ be an additively closed root sub-datum of a root datum $\Phi$. If $\Phi$ is simply laced then the dual $\Psi^\vee$ is also additively closed.
\end{Lem}
\begin{proof}
The lemma follows from the following assertion: if $\Phi$ is simply laced and $\alpha$ and $\beta$ are two distinct roots such that $\alpha+\beta$ is also a root, then $(\alpha+\beta)^\vee=\alpha^\vee+\beta^\vee$. (See \cite[10.2.2, p.177]{Springer}.)
\end{proof}

For our purposes, we need the following.
\begin{Lem}\label{L:additive_closure_dual}
Let $\Delta\subseteq R$ be a basis of the root system, and $\Pi\subseteq\Delta$ a subset. Then the dual $(\Pi^{ac})^\vee$ of the additive closure is also additively closed.
\end{Lem}
\begin{proof}
Elementary exercise.
\end{proof}

\subsection{Strongly orthogonal roots}
Two roots $\alpha$ and $\beta$ are said to be strongly orthogonal if
\[
\{\alpha, \beta\}^{ac}=\{\pm\alpha, \pm\beta\}.
\]
Strong orthogonality necessarily implies orthogonality. If $\alpha$ and $\beta$ are strongly orthogonal, then the subgroup generated by the additive closure $\{\alpha, \beta\}^{ac}$ is isogenous to $\SL_2\times\SL_2$. For example, two orthogonal long roots in $B_2$ are strongly orthogonal but two orthogonal short roots are not.

\subsection{Folding a root datum}
In this subsection, we review the method called ``folding a root datum", which is used to produce a map $\varphi:H\to G$ of reductive groups such that the map $H\to\varphi(H)$ is an isogeny. This method is a slight generalization of a well-known method (\cite[Section 10.3]{Springer}) and discussed in \cite[Section 4]{Knop_Schalke} in detail.

Let $(X(T), \Phi, X(T)^\vee, \Phi^\vee)$ be a root datum of a split reductive group $G$. Let $\Delta\subseteq \Phi$ be a set of simple roots, so that $(X(T), \Delta, X(T)^\vee, \Delta^\vee)$ is a based root datum.

\begin{Def}
Let $s:\Delta\to\Delta$ be an involution, which naturally induces an involution $s:\Delta^\vee\to\Delta^\vee$, so that $({^s\alpha})^\vee={^s(\alpha^\vee)}$, which we simply write $^s\alpha^\vee$. We call $s$ a {\it folding} if for all $\alpha, \beta\in\Delta$, we have
\begin{enumerate}[(a)]
\item $\la \alpha, {^s\alpha^\vee}\ra=0$, whenever $\alpha\neq {^s\alpha}$, and
\item $\la \alpha-{^s\alpha}, \beta^\vee+{^s\beta^\vee}\ra=0$.
\end{enumerate}
Note that the property $\la{^s\alpha}, {^s\beta^\vee}\ra=\la\alpha, \beta^\vee\ra$ implies (b).
\end{Def}

Let $A$ be a torus and $\varphi_A:A\to T$ a homomorphism with finite kernel, so that we have the map
\[
r:X(T)\longrightarrow X(A), \quad x\mapsto x\circ\varphi,
\]
given by the restriction via $\varphi_A$. We often write
\[
r(x)=\overline{x}.
\]
Note that the image of $r$ has a finite cokernel, which induces an injection
\[
r^\vee: X(A)^\vee\longrightarrow X(T)^\vee,
\]
which implies that $X(A)$ and $r^\vee(X(A)^\vee)$ are still dual to each other. We often identity $r^\vee(X(A)^\vee)$ with $X(A)^\vee$.

Now, let $s:\Delta\to\Delta$ be a folding. For each $\alpha\in \Delta$ we define
\[
\oalpha^\vee:=\begin{cases}\alpha^\vee&\text{if $\alpha={^s\alpha}$};\\
\alpha^\vee+{^s\alpha^\vee}&\text{otherwise},
\end{cases}
\]
and set
\[
\overline{\Delta}=r(\Delta)=:\{\oalpha\st\alpha\in\Delta\}
\qand\overline{\Delta}^\vee:=\{\oalpha^\vee\st\alpha\in\Delta\}.
\]
The following is the main theorem on folding.
\begin{Prop}\label{P:folding}
Let $s:\Delta\to\Delta$ be a folding such that
\begin{enumerate}[(i)]
\item $r(\alpha)=r({^s\alpha})$ for all $\alpha\in\Delta$, and
\item $\overline{\Delta}^\vee\subseteq X(A)^\vee$.
\end{enumerate}
Then the quadruple $(X(A), \overline{\Delta}, X(A)^\vee, \overline{\Delta}^\vee)$ is a reduced based root datum, and if $H$ is the corresponding split reductive group, there is a homomorphism
\[
\varphi:H\longrightarrow G
\]
with finite kernel which extends the map $\varphi_A:A\to T$ of the tori. If $\varphi_A$ is one-to-one, so is $\varphi$.
\end{Prop}
\begin{proof}
See \cite[Lemma 4.5]{Knop_Schalke}.
\end{proof}

It should be noted that, in the above proposition, since $r(\alpha)=r({^s\alpha})$, we can write
\[
\oalpha=\frac{1}{2}r(\alpha+{^s\alpha})
\]
for all $\oalpha\in\overline{\Delta}$.


\section{Root datum with involution}


We keep the notation of the previous section. In particular, we let $G$ be a split reductive group over an arbitrary filed $k$ of $\operatorname{char} k\neq 2$.

\subsection{Involution on root datum}

Fix a maximal torus $T\subseteq G$, and let
\[
\Phi(G, T)=(X(T), \Phi, X(T)^\vee, \Phi^\vee)
\]
be the root datum of $G$ with respect to $T$. In particular, it is reduced. We do not assume that $\Phi(G, T)$ is irreducible. We assume that $\Phi(G, T)$ is equipped with an involution $\theta$. 

By the isomorphism theorem, $\theta$ lifts to an automorphism $\tilde{\theta}$ on the group $G$. However, $\tilde{\theta}$ does not have to be an involution; namely it may happen that $\tilde{\theta}^2\neq 1$. The conditions for the lift $\tilde{\theta}$ to be an involution are discussed in \cite[Section 4.1, p.36]{Helminck_commuting_pair}. Yet, for our purposes, all we need is an involution on the root datum.

Certainly, the same involution can be viewed as an involution on the dual root datum
\[
\Phi(G, T)^\vee=(X(T)^\vee, \Phi^\vee, X(T), \Phi),
\]
which we denote by $\theta^\vee$ and call it the dual of $\theta$.

\subsection{Two tori}

A $k$-split torus $A\subseteq T$ is said to be $\theta$-invariant if $\theta(t)=t$ for all $t\in A$, and said to be $\theta$-split if $\theta(t)=t^{-1}$ for all $t\in A$. We set
\[
T^+:=\{t\in T\st \theta(t)=t\}^\circ\qand T^-:=\{t\in T\st \theta(t)=t^{-1}\}^\circ.
\]
It is easy to see that the map
\[
T^+\times T^-\longrightarrow T
\]
is an isogeny.

All of $T, T^+, T^-$ are closed under the action of $\theta$, and hence
$\theta$ naturally acts on the groups of rational characters
\[
X(T)=\Hom(T, \Gm),\quad
X(T^+)=\Hom(T^+, \Gm),\quad X(T^-)=\Hom(T^-, \Gm)
\]
and on the groups of rational cocharacters
\[
X(T)^\vee=\Hom(\Gm, T),\quad
X(T^+)^\vee=\Hom(\Gm, T^+),\quad X(T^-)^\vee=\Hom(\Gm, T^-)
\]
of the respective tori. The natural pairing
\[
\la-,-\ra:X(T)\times X(T)^\vee\longrightarrow \Z
\]
restricts to the pairings
\[
X(T^+)\times X(T^+)^\vee\longrightarrow\Z \qand X(T^-)\times X(T^-)^\vee\longrightarrow \Z,
\]
which we also denote by $\la-,-\ra$.

By the definition of the action of $\theta$ it is immediate that
\begin{equation}\label{E:theta_preserves_length}
\la \theta x, \theta x^\vee\ra=\la x, x^\vee\ra
\end{equation}
for all $x\in X(T)$ and $x^\vee\in X(T)^\vee$. In particular, $\theta$ preserves the lengths of roots and coroots.

\subsection{Admissible involution}
An involution $\theta$ on the root datum $\Phi(G, T)$ is called {\it admissible} if the corresponding automorphism $\tilde{\theta}$ on $G$ is also an involution (namely $\theta$ can be lifted) and $T^-$ is a {\it maximal} $\tilde{\theta}$-split torus, namely of maximum rank among the $\tilde{\theta}$-split tori. An involution $\theta$ fails to be admissible in two ways: $\theta$ cannot be lifted, or even if $\theta$ is lifted, $T^-$ is not a maximal $\tilde{\theta}$-split torus. For example, the involution to be discussed in Example \ref{Ex:A_2} is not admissible even though it can be lifted. (See \cite[Proposition 4.4 and Corollary 4.5, p.37]{Helminck_commuting_pair} more on conditions of liftability and admissibility.)

To clarify the difference between an admissible involution and an involution which can be merely lifted, let us make the following observation. If $\theta$ is lifted to an involution $\tilde{\theta}$ on $G$, one can find a possibly different torus $S\subseteq G$ so that the torus $S^-$ is a maximal $\tilde{\theta}$-split torus. Then $\tilde{\theta}$ descends to an involution on the root datum $\Phi(G, S)$. This involution on $\Phi(G, S)$ is admissible, yet it is not equivalent to the original involution $\theta$ on $\Phi(G, T)$ unless the original $\theta$ is admissible.

For our purposes, we might as well focus on those involutions whose dual involutions are admissible because we are interested in the dual involution of an involution on a reductive group $G$ over $F$. However, we will develop our construction as in general as possible. Hence we will not make any such assumption, though later we will have to exclude some of the involutions whose duals are not admissible.

\subsection{$\theta$-basis}
We set
\[
\Phi_0=\{\alpha\in\Phi\st \theta(\alpha)=\alpha\}\qand \Phi_0^{\vee}=\{\alpha^\vee\in\Phi\st \theta(\alpha^\vee)=\alpha^\vee\},
\]
namely the set of all $\theta$-invariant roots and coroots, respectively.

It is well-known that there is a basis $\Delta\subseteq\Phi$  so that the corresponding ordering has the property
\begin{equation}\label{E:positivity}
\alpha>0\quad\text{and}\quad \alpha\notin\Phi_0
\quad\Longrightarrow\quad\theta(\alpha)<0.
\end{equation}
We call this ordering a $\theta$-order and such basis $\Delta$ a $\theta$-basis. We fix such $\Delta$ and let $\Phi^+$ be the set of positive roots with respect to this ordering.

\subsection{$\theta$-invariant root sub-datum}

One can readily see that the quadruple
\[
\Phi(G, T)_0:=(X(T),\Phi_0,X(T)^\vee,\Phi_0^{\vee})
\]
is an additively closed root sub-datum of $\Phi(G, T)$ where the sets of simple roots and coroots are, respectively,
\[
\Delta_0:=\Delta\cap\Phi_0\qand\Delta_0^\vee:=\Delta^\vee\cap\Phi_0^{\vee}.
\]
We let $G^{++}$ be the split reductive group generated by this root datum inside $G$ (Proposition \ref{P:subgroup_generated_by_sub-datum}), so that we have the natural inclusion
\[
G^{++}\subseteq G,
\]
and the root datum of $G^{++}$ is $\Phi(G, T)_0$. Note that $G^{++}$ is the subgroup generated by $T$ and the root subgroups $U_\alpha$ for all $\alpha\in\Phi_0$.

Now, consider the restriction map $r:X(T)\to X(T^+)$. Then the induced map $\Delta_0\to r(\Delta_0)$ is a bijection, via which we often identify these two sets, and we have the identity $r^\vee(\Delta_0^\vee)=\Delta_0^\vee$. Hence by choosing $s:\Delta_0\to\Delta_0$ to be the identity map, which is certainly a folding, we can apply Proposition \ref{P:folding} and obtain a based root datum
\[
\Phi(G, T^+):=(X(T^+),\Delta_0,X(T^+)^\vee,\Delta_0^\vee)
\]
and an inclusion
\[
G^{+}\subseteq G,
\]
where $G^{+}$ is the reductive group whose root datum is $\Phi(G, T^+)$. Note that the group $G^+$ is generated by $T^+$ and the root subgroups $U_\alpha$ for all $\alpha\in\Phi_0$, and the map $G^{+}\to G$ given by folding is an inclusion. We call the root datum $\Phi(G, T^+)$ the $\theta$-invariant root datum.

Let us note that the group $G^{++}$ plays only an auxiliary role to make use of the method of folding and to obtain the inclusion $G^+\subseteq G$. Also note that the main difference between $G^{++}$ and $G^{+}$ is that $T\subseteq G^{++}$ and $T^+\subseteq G^+$ (but $T\nsubseteq G^+$ in general).

Let us mention the following lemma.
\begin{Lem}\label{L:invariant_non-invariant_roots_orthogonal}
For all $\alpha\in\Phi$ and $\gamma\in\Phi_0$, we have
\[
\la\alpha-\theta\alpha, \gamma^\vee\ra=0.
\]
\end{Lem}
\begin{proof}
\[
\la\alpha-\theta\alpha, \gamma^\vee\ra=\la\alpha, \gamma^\vee\ra-\la\theta\alpha, \gamma^\vee\ra=\la\alpha, \gamma^\vee\ra-\la\alpha, \theta\gamma^\vee\ra=\la\alpha, \gamma^\vee\ra-\la\alpha, \gamma^\vee\ra=0.
\]
\end{proof}

\subsection{A lemma}

Let us then quote the following lemma, which will be frequently used in this paper.
\begin{Lem}\label{L:theta^*}
Let $w_0$ be the longest element in the Weyl group of the $\theta$-invariant root datum $\Phi(G, T^+)$. For each involution $\theta$, there exists a possibly trivial automorphism $\theta^*$ of the based root datum $(X(T), \Delta, X(T)^\vee, \Delta^\vee)$ such that for each root $\alpha\in\Phi$, one can write
\[
-\theta\alpha=(\theta^*\circ w_0)\alpha.
\]
In particular, for each $\alpha\in\Delta\smallsetminus\Delta_0$ we have
\[
-\theta\alpha=\theta^*\alpha+\gamma\quad\text{for some $\gamma\in\Span_\Z(\Delta_0)$}.
\]

Here, if $\Phi(G, T)$ is irreducible, then $\theta^*$ is an automorphism of the Dynkin diagram. If it is of the form $\Phi_1\coprod\Phi_2$ with $\Phi_1, \Phi_2$ irreducible and $\theta(\Phi_1)=\Phi_2$, then $\theta^*$ exchanges the Dynkin diagrams of $\Phi_1$ and $\Phi_2$, in which case $\Phi_0=\emptyset$, so $w_0=\operatorname{id}$ and $\theta=-\theta^*$.
\end{Lem}
\begin{proof}
See \cite[2.8]{Helminck_commuting_pair} and  \cite[1.7]{Helminck-Helmink}. It should be notated that in \cite{Helminck-Helmink}, this lemma is stated as $\gamma\in\Phi_0$ (in stead of $\gamma\in\Span_\Z(\Delta_0)$). Apparently, this is a typographical error.
\end{proof}

\subsection{Principal $\SL_2$ of $G^+$}
Let
\[
\varphi^+:\SL_2\longrightarrow G^+
\]
be the principal $\SL_2$-homomorphism of $G^+$, so that
\[
\varphi^+(\begin{pmatrix}t&\\&t^{-1}\end{pmatrix})=2\rho_0^\vee(t),
\]
where
\[
2\rho_0^\vee=\sum_{\gamma\in\Phi_0^+}\gamma^\vee
\]
is the sum of $\theta$-invariant positive coroots. (For the principal $\SL_2$ homomorphism, see, for example, \cite[Section 2]{Gross}.)

Lemma \ref{L:invariant_non-invariant_roots_orthogonal} gives us
\begin{equation}\label{E:rho_0_orthogonal}
\la\alpha-\theta\alpha, 2\rho_0^\vee\ra=0\quad\text{for all $\alpha\in\Phi$}.
\end{equation}

Noting that $2\rho_0^\vee:\Gm\to T^+$, we set
\[
T^0:=\text{the image of $2\rho_0^\vee$},
\]
which is a 1-dimensional subtorus of $T^+$.

\subsection{$\theta$-split root datum}
Next, we consider the root datum associated with the $\theta$-split torus $T^-$ by the restriction map
\[
p:X(T)\to X(T^-).
\]
Let
\[
\overline{\Phi}=p(\Phi)\smallsetminus\{0\}=p(\Phi\smallsetminus\Phi_0).
\]
It has been proven by Helminck and Wang (\cite{Helmink-Wang}) that, by tensoring with $\Q$, $\overline{\Phi}$ is a root system in $X(T^-)\otimes \Q$ with a basis
\[
\overline{\Delta}:=p(\Delta)\smallsetminus\{0\}=p(\Delta\smallsetminus\Delta_0).
\]
This gives rise to a root datum
\begin{equation}\label{E:theta-split-root-datum}
\overline{\Phi}(G, T^-):=(X(T^-),\overline{\Phi}, X(T^-)^\vee, \overline{\Phi}^\vee),
\end{equation}
which we call the $\theta$-split root datum.

The $\theta$-split root datum $\overline{\Phi}(G, T^-)$ is not always reduced as in the following lemma.
\begin{Lem}\label{L:reduced_split_root_datum}
If there is a root $\alpha\in\Phi\smallsetminus\Phi_0$ such that $\alpha-\theta\alpha$ is also a root, then the $\theta$-split root datum $\overline{\Phi}(G, T^-)$ is non-reduced.
\end{Lem}
\begin{proof}
If $\alpha-\theta\alpha$ is a root, then $\overline{\alpha-\theta\alpha}=\oalpha+\oalpha=2\oalpha$ because $\overline{-\theta\alpha}=\oalpha$.
\end{proof}

\begin{Rmk}
The converse of the lemma also holds. Yet, since we will not need the converse for our purposes, we will leave the proof to the reader.
\end{Rmk}

Assume $\overline{\Phi}(G, T^-)$ is reduced. Our major goal is to produce a subgroup $G^-\subseteq G$ such that its root datum is precisely $\overline{\Phi}(G, T^-)$ and show that $G^-$ commutes with the image $\varphi^+(G^+)$ of the principal $\SL_2$ of $G^+$. If $\overline{\Phi}(G, T^-)$ is not reduced, then  it has an irreducible component of type $BC_n$ and we will make it reduced by discarding the ``short roots". Then the root datum of $G^-$ is this reduced root datum.


\section{Dynkin diagram of $\theta$ and some examples}


We keep the notation and assumption from the previous section. In particular, $\Phi(G, T)$ is a root datum of a reductive group over $k$ with an involution $\theta$ on $\Phi(G, T)$. We assume that a $\theta$-basis $\Delta$ is chosen, so that $(X(T), \Delta, X(T)^\vee, \Delta^\vee)$ is a based root datum with a $\theta$-order. We let $\Delta_0$ be the set of $\theta$-invariant simple roots.

In this section, we review the notion of Dynkin diagram of $\theta$, which is also called the index of $\theta$ in \cite{Helminck_commuting_pair}, and discuss some instructive examples. We also define the concept of the diagram of $\alpha$ for $\alpha\in\Delta\smallsetminus\Delta_0$.

\subsection{Dynkin diagram of $\theta$}
Assume $\Phi(G, T)$ is irreducible, so that the corresponding $\theta^*$ of Lemma \ref{L:theta^*} is an automorphism of the Dynkin diagram of $\Phi(G, T)$. For the given $\theta$-basis, we draw the corresponding Dynkin diagram with black dots for the simple roots in $\Delta_0$ and white dots for the simple roots in $\Delta\smallsetminus\Delta_0$. Also we sometimes incorporate the action of $\theta^*$ by using a two headed arrow, if necessary. An example of such Dynkin diagram is
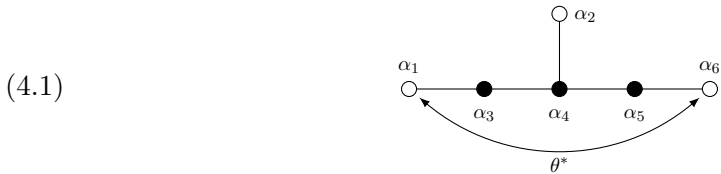
\begin{equation}\label{Ex:E_6}
\dynkin[labels={,\alpha_2,\alpha_3,\alpha_4,\alpha_5,}, labels*={\alpha_1,,,,,\alpha_6}, involutions={[in=-140, out=-40]1<below>[\theta^*]6}]E{oo***o}
\end{equation}
Here, the three roots in the middle are $\theta$-invariant, namely
\[
\Delta_0=\{\alpha_3, \alpha_4, \alpha_5\},
\]
and $\theta^*$ is the diagram automorphism that flips the diagram along the middle line, namely
\[
\theta^*\alpha_1=\alpha_6\qand \theta^*\alpha_3=\alpha_5,
\]
and the other roots are fixed. We call such diagram the Dynkin diagram of $\theta$. Up to isogeny, the conjugacy classes of the involutions are classified by such Dynkin diagrams. (See Sections 2, 3 and 4 of \cite{Helminck_commuting_pair} for more detail, though the Dynkin diagram of $\theta$ is called the index of $\theta$ there.)

We say that the involution $\theta$ is of rank one if the $\theta$-split root datum $\overline{\Phi}(G, T^-)$ is of rank one in the sense that $\overline{\Phi}$ spans only one dimensional space. All the rank one involutions (up to isogeny) are listed in Table \ref{T:table} in Appendix \ref{A:table}, which is essentially taken from \cite[\textsc{Table} I, p.39]{Helminck_commuting_pair}. (Let us note that in \cite{Helminck_commuting_pair}, a rank one involution is called a restricted rank one involution.)

For each $\alpha\in\Delta\smallsetminus\Delta_0$, we set
\[
\Phi_\alpha=\text{ additive closure of $\{\alpha, \theta^*\alpha\}\cup\Delta_0$}.
\]
This gives rise to a root sub-datum of $\Phi(G, T)$ which is closed under the involution $\theta$, and the set $\{\alpha, \theta^*\alpha\}\cup\Delta_0$ is a $\theta$-basis.

\begin{Def}\label{D:diagram_of_root}
For $\alpha\in\Delta\smallsetminus\Delta_0$, we define the {\it diagram of $\alpha$} to be the Dynkin diagram of the irreducible component of $\Phi_\alpha$ containing $\alpha$, with the proviso that we consider the $A_1\times A_1$ diagram
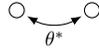
\[
\dynkin[involutions={1<below>[\theta^*]2}, edge/.style={white}]A{oo}
\]
as irreducible. These diagrams are precisely the ones listed in Table \ref{T:table}.
\end{Def}

The diagram of each simple root $\alpha$ is in a way a building block of a given involution, and the involution on a root datum is to be analyzed in terms of the diagrams of the simple roots. For instance, in the above $E_6$ example \eqref{Ex:E_6}, the diagram consists of the following two subdiagrams
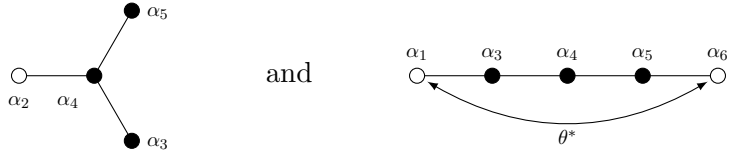
\[
\dynkin[labels={\alpha_2, \alpha_4, \alpha_5, \alpha_3}]D{o***}\qquad\qand\qquad
\dynkin[labels*={\alpha_1, \alpha_3, \alpha_4, \alpha_5, \alpha_6}, involutions={1<below>[\theta^*]5}] A{o***o}
\]
where the first one is the diagram of $\alpha_2$ and the second one is that of $\alpha_1$ as well as of $\alpha_6$. In light of Lemma \ref{L:theta^*}, one can then tell how $\theta$ acts on each simple root.

\subsection{Some examples of Dynkin diagrams}\label{SS:examples_Dynkin_diagrams}
Let us discuss some of the important examples of Dynkin diagrams of $\theta$ for rank one cases. All the rank one cases are discussed in Appendix \ref{A:table} and summarized in Table \ref{T:table}.

\begin{Ex}\label{Ex:A_1}
Let us consider the simplest case
\[
\dynkin[labels={\alpha}] A{o}
\]
of type  $A_1$. This is the case if and only if $\theta\alpha=-\alpha$. If an involution has this type of simple root $\alpha$, then $\alpha$ is not connected to any of the $\theta$-invariant roots in the Dynkin diagram. For later purposes, let us mention that $\alpha-\theta\alpha=2\alpha$, which is not a root, but
\[
\frac{1}{2}(\alpha-\theta\alpha)\in\Phi.
\]
\end{Ex}

\begin{Ex}\label{Ex:A_1_x_A_1}
Consider the following $A_1\times A_1$ example
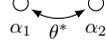
\[
\dynkin[labels={\alpha_1,\alpha_2}, involutions={1<below>[\theta^*]2}, edge/.style={white}]A{oo}
\]
Certainly in this case, $-\theta\alpha_1=\theta^*\alpha_1=\alpha_2$ and
\[
\alpha_i-\theta\alpha_i\notin\Phi\qand \frac{1}{2}(\alpha_i-\theta\alpha_i)\notin\Phi
\]
for $i=1,2$.
\end{Ex}

\begin{Ex}\label{Ex:A_2}
Consider the following $A_2$ example
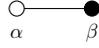
\[
\dynkin[labels={\alpha, \beta}] A{o*}
\]
Since $-\theta\alpha=w_0\alpha$, where $w_0$ is the longest Weyl element of the $A_1$ system consisting of the $\theta$-invariant root $\beta$, namely the simple reflection associated with $\beta$, we have
\[
-\theta\alpha=\alpha+\beta.
\]
For later purposes, let us mention that, since $\alpha-\theta\alpha=2\alpha+\beta$, we know
\[
\alpha-\theta\alpha\notin\Phi\qand \frac{1}{2}(\alpha-\theta\alpha)\notin\Phi.
\]
This is the only involution of type $A$ which is not admissible.
\end{Ex}

\begin{Ex}\label{Ex:A_n}
Let us consider
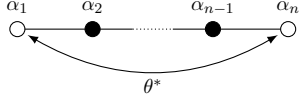
\[
\dynkin[labels*={\alpha_1, \alpha_2,\alpha_{n-1}, \alpha_{n}}, involutions={1<below>[\theta^*]4}]A{o*.*o}
\]
where
\[
\alpha_1=e_1-e_2,\quad \alpha_2=e_2-e_3,\quad\cdots,\quad\alpha_{n}=e_{n}-e_{n+1}.
\]
Then $\Delta_0=\{\alpha_2,\cdots,\alpha_{n-1}\}$ and $\Phi_0$ is of type $A_{n-2}$. The longest element $w_0$ of the $\theta$-invariant system $\Phi_0$ acts as
\[
w_0e_1=e_1,\quad w_0e_{n+1}=e_{n+1},\quad w_0e_2=e_{n},\quad w_0e_3=e_{n-1},\quad\cdots,
\]
and the diagram involution $\theta^*$ acts as
\[
\theta^*e_1=-e_{n+1}\qand \theta^*e_i=-w_0e_i\quad (i=2,\dots, n).
\]
So
\begin{align*}
-\theta\alpha_1&=\theta^*w_0(e_1-e_2)=\theta^*(e_1-e_n)=e_2-e_{n+1}\\
-\theta\alpha_n&=\theta^*w_0(e_n-e_{n+1})=\theta^*(e_2-e_{n+1})=e_1-e_{n}.
\end{align*}
Note that
\[
\alpha_1-\theta\alpha_1=\alpha_n-\theta\alpha_n=e_1-e_{n+1}\in\Phi.
\]
\end{Ex}

\begin{Ex}\label{Ex:A_3}
Consider the diagram
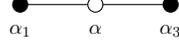
\[
\dynkin[labels={\alpha_1,\alpha,\alpha_3}] A{*o*}
\]
where
\[
\alpha_1=e_1-e_2,\quad \alpha=e_2-e_3,\quad \alpha_3=e_3-e_4.
\]
One can see that
\[
w_0e_1=e_2\qand w_0e_3=e_4,
\]
so
\[
-\theta\alpha=w_0(e_2-e_3)=e_1-e_4.
\]
Then
\[
\alpha-\theta\alpha=(e_1-e_3)+(e_2-e_4),
\]
so
\[
\alpha-\theta\alpha\notin\Phi\qand \frac{1}{2}(\alpha-\theta\alpha)\notin\Phi.
\]
\end{Ex}

Though this is an example of root system of type $A_3$, for our purposes, it should be viewed as a special case of the following type $D_n$ example.

\begin{Ex}\label{Ex:D_n}
Consider the involution with the Dynkin diagram
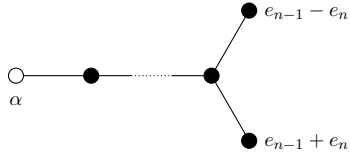
\[
\dynkin[labels={\alpha,,,e_{n-1}-e_n, e_{n-1}+e_n}]D{o*.***}
\]
where $\alpha=e_1-e_2$. Noting that $-\theta\alpha=\theta^*w_0\alpha$, where $w_0$ is the longest element of the Weyl group of type $D_{n-1}$, one can readily compute that
\[
-\theta\alpha=e_1+e_2.
\]
Then
\[
\alpha-\theta\alpha=2e_1\notin\Phi.
\]
Here, for later purposes, we rewrite it as
\[
\alpha-\theta\alpha=(e_1-e_n)+(e_1+e_n).
\]
\end{Ex}

\begin{Ex}\label{Ex:B_n}
Consider
\[
\dynkin[labels={\alpha,\alpha_2,\alpha_3,\alpha_{n-1},\alpha_n}] B{o**.**}
\]
where
\[
\alpha=e_1-e_2,\quad\alpha_2=e_2-e_3,\quad\dots,\quad\alpha_{n-1}=e_{n-1}-e_n,\quad\alpha_n=e_n.
\]
The Weyl group element $w_0$ acts as
\[
w_0 e_1=e_1,\qquad w_0e_i=-e_i\quad (2\leq i\leq n),
\]
and $-\theta=w_0$. Hence
\[
\alpha-\theta\alpha=e_1-e_2+w_0(e_1-e_2)=2e_1,
\]
so
\[
\frac{1}{2}(\alpha-\theta\alpha)\in\Phi.
\]
\end{Ex}

\subsection{Properties of non-invariant roots}
For each $\alpha\in\Delta\smallsetminus\Delta_0$, depending on the behavior of $\alpha-\theta\alpha$, we say
\[
\text{$\alpha$ is of}\; \begin{cases}\text{type 1}&\text{if $\alpha-\theta\alpha\in\Phi$};\\
\text{type 2}&\text{if $\frac{1}{2}(\alpha-\theta\alpha)\in\Phi$};\\
\text{type 3}&\text{otherwise}.\end{cases}
\]
It is immediate that these three cases are mutually exclusive because our root system $\Phi$ is reduced.

Most of the simple roots $\alpha\in\Delta\smallsetminus\Delta_0$ are actually of type 1 or type 2, and moreover if a root is of type 2 with $-\theta\alpha\neq \alpha$, then it essentially comes from the $B_n$ case discussed in Example \ref{Ex:B_n}, and if a root is of type 3, it essentially comes from the $A_2$ case (Example \ref{Ex:A_2}), the $A_1\times A_1$ case (Example \ref{Ex:A_n}), or the $D_n$ case (Example \ref{Ex:D_n}). To be precise, we have the following.

\begin{Prop}\label{P:types_of_simple_roots}
Let $\alpha\in\Delta\smallsetminus\Delta_0$ be not of type 1. Then the following hold, where the white dot corresponds to $\alpha$.
\begin{enumerate}[(a)]
\item $\alpha$ is of type 2 if and only if the diagram of $\alpha$ is one of the following two
\[
\dynkin A{o}\qquad\qquad
\dynkin B{o*.**}
\]
\item $\alpha$ is of type 3 if and only if the diagram of $\alpha$ is one of the following three
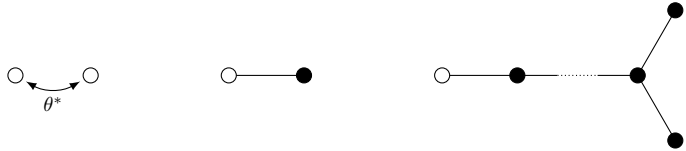
\[
\dynkin[involutions={1<below>[\theta^*]2}, edge/.style={white}]A{oo}\qquad\qquad
\dynkin A{o*}\qquad\qquad
\dynkin D{o*.***}
\]
\end{enumerate}
\end{Prop}
\begin{proof}
As mentioned already, the diagram of $\alpha$ is one of the 18 cases listed in Table \ref{T:table} in Appendix \ref{A:table}. One can prove the proposition by checking all the 18 cases, which is done in Appendix \ref{A:table}.
\end{proof}


\section{Construction of $G^-$ by folding}


We still keep the notation and assumption. In particular, $\theta$ is an involution on the root datum $\Phi(G, T)$. In this section, we construct the subgroup $G^-\subseteq G$ whose root datum is the $\theta$-split root datum $\overline{\Phi}(G, T^-)$ (with the short roots discarded if it is non-reduced).

\subsection{On roots of type 3}\label{SS:Type_3_root}
Assume $\alpha\in\Delta\smallsetminus\Delta_0$ is of type 3, so that its diagram is one of the three in Proposition \ref{P:types_of_simple_roots} (b).
\begin{enumerate}[$\bullet$]
\item If the diagram of $\alpha$ is
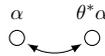
\[
\dynkin[labels*={\alpha, \theta^*\alpha}, involutions={1<above>2}, edge/.style={white}]A{oo}
\]
then
\[
-\theta\alpha=\theta^*\alpha,
\]
so $\alpha-\theta\alpha=\alpha+\theta^*\alpha$, which is not a root. We set
\[
\alpha_1=\alpha\qand\alpha_2=\theta^*\alpha,
\]
so
\[
-\theta\alpha_1=\alpha_2\qand \la\alpha_1,\alpha_2^\vee\ra=0.
\]

\item If the diagram of $\alpha$ is
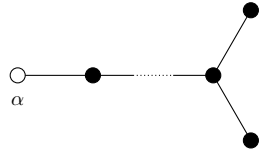
\[
\dynkin[labels={\alpha}]D{o*.***}
\]
then
\[
\alpha=e_1-e_2\qand -\theta\alpha=e_1+e_2,
\]
so $\alpha-\theta\alpha=2e_1$, which is not a root. We set
\[
\alpha_1=e_1-e_n\qand \alpha_2=e_1+e_n,
\]
so
\[
-\theta\alpha_1=\alpha_2\qand \la\alpha_1,\alpha_2^\vee\ra=0.
\]
This choice of $\alpha_1$ and $\alpha_2$ is made so that we have the following.
\begin{Lem}\label{L:type3_orthogonal_to_rho_0}
Recall that $2\rho_0^\vee$ is the sum of positive $\theta$-invariant coroots. We have
\begin{equation}
\la\alpha_1,2\rho_0^\vee\ra=\la\alpha_2,2\rho_0^\vee\ra=0.
\end{equation}
\end{Lem}
\begin{proof}
We may assume that $2\rho_0^\vee$ is the sum of positive coroots of only the diagram of $\alpha$ because all the coroots that are not in the diagram of $\alpha$ are clearly orthogonal to $\alpha$. Then the positive coroots are of the form $\{e_i^\vee\pm e_j^\vee\st 2\leq i<j\leq n\}$. One can then see that
\[
2\rho_0^\vee=2(n-2)e_2^\vee+2(n-3)e_3^\vee+\cdots+2e_{n-1}^\vee,
\]
so $2\rho_0^\vee$ does not have a term of $e_1^\vee$ or $e_n^\vee$. The lemma follows.
\end{proof}

\item Finally, we assume that the diagram of $\alpha$ never takes the form
\begin{equation}\label{E:excluded_diagram_type3}
\dynkin A{o*}
\end{equation}
Namely, we assume that the Dynkin diagram of our $\theta$ never has this $A_2$ diagram as a subdiagram. But this is not a serious assumption for our purposes in that this diagram never appears in the dual involution of an admissible involution on a reductive group.
\end{enumerate}

\quad

We can summarize this subsection as follows.
\begin{Prop}\label{P:decomposition_type_3_root}
Assume that $\alpha\in\Delta\smallsetminus\Delta_0$ is of type 3 and not of the form $\dynkin[edge-length=.5cm, root-radius=.05cm]A{o*}$. Then there exist a unique pair of roots $\{\alpha_1,\alpha_2\}$ such that
\begin{enumerate}[(a)]
\item $\alpha-\theta\alpha=\alpha_1+\alpha_2$;
\item $\alpha_1$ and $\alpha_2$ are strongly orthogonal;
\item $\la \alpha_1, 2\rho_0^\vee\ra=\la \alpha_2, 2\rho_0^\vee\ra=0$.
\end{enumerate}
\end{Prop}
\begin{proof}
We have not proven the strong orthogonality and uniqueness. But they are almost immediate and left to the reader.
\end{proof}

\subsection{A root sub-datum}
In this subsection, we introduce a root sub-datum of $\Phi(G, T)$, which we will ``fold" to construct the subgroup $G^-\subseteq G$, assuming that the diagram of any $\alpha\in\Delta\smallsetminus\Delta_0$ is not of the form $\dynkin[edge-length=.5cm, root-radius=.05cm]A{o*}$.

First, we set
\[
\Sigma_\theta=\bigcup_{\text{$\alpha$ of type 1}}\{\alpha-\theta\alpha\}
\cup\bigcup_{\text{$\alpha$ of type 2}}\{\frac{1}{2}(\alpha-\theta\alpha)\}
\cup\bigcup_{\text{$\alpha$ of type 3}}\{\alpha_1, \alpha_2\},
\]
where the union is over only simple roots. Notice that all the elements in $\Sigma_\theta$ are positive roots, and the unions are all disjoint. Furthermore, the involution $-\theta$ acts on $\Sigma_\theta$ so that each one of the sets $\{\alpha-\theta\alpha\}$, $\{\frac{1}{2}(\alpha-\theta\alpha)\}$ and $\{\alpha_1, \alpha_2\}$ is a single orbit, where for the type 3 case, $-\theta$ interchanges $\alpha_1$ and $\alpha_2$.

Set
\[
\Phi(G, T)_\theta:=(X(T), \Sigma_\theta^{ac}, X(T)^\vee, \Sigma_\theta^{ac\vee}),
\]
which is an additively closed root sub-datum of $\Phi(G, T)$ by Lemma \ref{L:additive_closure_is_root_datum}.

\begin{Lem}
The set $\Sigma_\theta$ is a basis of $\Phi(G, T)_\theta$.
\end{Lem}
\begin{proof}
By \cite[Lemma 3.3]{Knop_Schalke}, it suffices to check that for any two $\alpha', \beta'\in\Sigma_\theta$ we have $\alpha'-\beta'\notin\Phi^+$. Certainly, we may assume $\Phi(G, T)$ is irreducible. Then the following are the possibilities of $\alpha'$ and $\beta'$:
\[
\alpha'=\alpha-\theta\alpha,\; \frac{1}{2}(\alpha-\theta\alpha),\;\text{or}\;\alpha+\mu,
\qand
\beta'=\beta-\theta\beta,\; \frac{1}{2}(\beta-\theta\beta),\;\text{or}\;\beta+\nu,
\]
where $\alpha$ and $\beta$ are simple roots of type 1, 2 and 3, respectively, and $\mu, \nu\in\Span_{\Z}(\Delta_0)$. Certainly, we may assume that $\alpha\neq\beta$, because otherwise $\alpha'=\beta'$ and hence $\alpha'-\beta'=0\notin\Phi^+$. Also from Lemma \ref{L:theta^*} we have
\[
-\theta\alpha=\theta^*\alpha+\gamma\qand -\theta\beta=\theta^*\beta+\delta
\]
for some $\gamma, \delta\in\Span_{\Z}(\Delta_0)$, where $\theta^*$ is a diagram automorphism.

We will show that $\alpha'-\beta'$ is not even a root, let alone a positive one. Also in the following argument, we often use the well-known fact that a root is always either a positive linear combination of simple roots or a negative linear combination.

Assume $\alpha'=\alpha-\theta\alpha$ and $\beta'=\beta-\theta\beta$. Then we have
\[
\alpha'-\beta'=\alpha+\theta^*\alpha+\gamma-\beta-\theta^*\beta-\delta.
\]
Since $\theta^*\alpha\in\Delta\smallsetminus\Delta_0$ and $\alpha\neq\beta$, in order for $\alpha'-\beta'$ to be a root, we must have $\theta^*\alpha=\beta$, which would give us $\alpha'-\beta'=\gamma-\delta$. So $\theta(\alpha'-\beta')=\alpha'-\beta'$. But at the same time $\theta(\alpha'-\beta')=-(\alpha'-\beta')$ because $\theta\alpha'=-\alpha'$ and $\theta\beta'=-\beta'$, so $\alpha'-\beta'=0$, which is not a root.

Assume  $\alpha'=\alpha-\theta\alpha$ and $\beta'=\frac{1}{2}(\beta-\theta\beta)$. There are two different cases for $\beta'$: either $\beta=-\theta\beta$ or not. Assume $\beta=-\theta\beta$, in which case $\beta'=\beta$. Then
\[
\alpha'-\beta'=\alpha+\theta^*\alpha+\gamma-\beta.
\]
Since $\alpha\neq\beta$, in order for this to be a root, we must have $\theta^*\alpha=\beta$. But then $\beta=-\theta\beta=\theta^*\beta+\delta=\alpha+\delta$, which must imply $\beta=\alpha$, which is a contradiction. Hence $\alpha'-\beta'$ is not a root.

Next, assume $\beta\neq-\theta\beta$. In this case, we already know that $\theta^*\beta=\beta$ (Proposition \ref{P:types_of_simple_roots}), and hence
\[
\alpha'-\beta'=\alpha+\theta^*\alpha+\gamma-\beta-\frac{1}{2}\delta.
\]
In order for this to be a root, we must have $\theta^*\alpha=\beta$. But since $\theta^*\beta=\beta$, we must have $\alpha=\beta$, which is a contradiction. Hence $\alpha'-\beta'$ is not a root.

Assume $\alpha'=\alpha-\theta\alpha$ and $\beta'=\beta+\nu$. Then
\[
\alpha'-\beta'=\alpha+\theta^*\alpha+\gamma-\beta-\nu.
\]
In order for this to be a root, we must have $\theta^*\alpha=\beta$. If the diagram of $\beta$ is $A_1\times A_1$, then $\theta^*\alpha=\beta$ would imply that $\alpha$ is not of type 1. If the diagram of $\beta$ is $D_n$ then $\theta^*\beta=\beta$, which gives $\alpha=\beta$, a contradiction. Hence $\alpha'-\beta'$ is not a root.

Next, assume $\alpha'=\frac{1}{2}(\alpha-\theta\alpha)$, in which there are two subcases: either $-\theta\alpha=\alpha$ or not. By symmetry with the cases already considered above, we may assume either $\beta'=\frac{1}{2}(\beta-\theta\beta)$ or $\beta'=\beta+\nu$. Assume $-\theta\alpha=\alpha$ and $-\theta\beta=\beta$. Then $\alpha'-\beta'=\alpha-\beta$, which is not a root. Assume $-\theta\alpha\neq \alpha$ and $-\theta\beta\neq \beta$ but $\beta'=\frac{1}{2}(\beta-\theta\beta)$. Then
\[
\alpha'-\beta'=\alpha-\beta+\frac{1}{2}\gamma-\frac{1}{2}\delta.
\]
In order for this to be a root, we must have $\alpha=\beta$, which is a contradiction.

Next, assume $\alpha'=\frac{1}{2}(\alpha-\theta\alpha)$ and $\beta'=\beta+\nu$. Again by considering the two subcases for $\alpha$, one can derive a contradiction for each case.

The remaining case is when $\alpha'=\alpha+\mu$ and $\beta'=\beta+\nu$. But again by the same reasoning, one can see that $\alpha'-\beta'$ is not a root.
\end{proof}

We thus have the following.

\begin{Prop}
There is a subgroup $G^{--}\subseteq G$ whose root datum is $\Phi(G, T)_\theta$ with a basis $\Sigma_\theta$.
\end{Prop}
\begin{proof}
Apply Proposition \ref{P:subgroup_generated_by_sub-datum} to $\Phi(G, T)_\theta$
\end{proof}

\subsection{Folding $\Phi(G, T)_\theta$}
Now, we will fold the above constructed root datum $\Phi(G, T)_\theta$. Let us define
\[
s:\Sigma_\theta\longrightarrow\Sigma_\theta,\quad \alpha'\mapsto -\theta\alpha',
\]
namely
\begin{align*}
^s(\alpha-\theta\alpha)&=\alpha-\theta\alpha;\\
^s(\frac{1}{2}(\alpha-\theta\alpha))&=\frac{1}{2}(\alpha-\theta\alpha);\\
^s(\alpha_1)&=\alpha_2,
\end{align*}
depending on the type of $\alpha$. This is certainly an involution. For each
\[
\alpha'=\begin{cases}\alpha-\theta\alpha, &\text{if $\alpha$ is of type 1};\\\frac{1}{2}(\alpha-\theta\alpha),&\text{if $\alpha$ is of type 2};\\
 \alpha_i,&\text{if $\alpha$ is of type 3},\end{cases}
\]
we have
\[
\alpha'+{^s\alpha'}=\begin{cases}2(\alpha-\theta\alpha)&\text{if $\alpha$ is of type 1};\\
\alpha-\theta\alpha&\text{if $\alpha$ is of type 2 or 3},\end{cases}
\]
and
\[
{\alpha'}^\vee+{^s\alpha'}^\vee=\begin{cases}2(\alpha-\theta\alpha)^\vee&\text{if $\alpha$ is of type 1};\\(\alpha-\theta\alpha)^\vee&\text{if $\alpha$ is of type 2 or 3}.\end{cases}
\]

\begin{Lem}
The above involution is a folding.
\end{Lem}
\begin{proof}
To check Condition (a) of folding, we may assume $\alpha$ is of type 3. But we know
\[
\la\alpha_1, {^s\alpha_1^\vee}\ra=\la \alpha_1, \alpha_2^\vee\ra=0
\]
by Proposition \ref{P:decomposition_type_3_root}.

Condition (b) of folding holds because in our case $s=-\theta$ and hence $\la{^s\alpha}, {^s\beta}^\vee\ra=\la \alpha, \beta^\vee\ra$ for all roots $\alpha, \beta$ by \eqref{E:theta_preserves_length}.
\end{proof}

\begin{Lem}
Let $r:X(T)\to X(T^-)$ be the restriction map. Then $r(\alpha')=r({^s\alpha'})$ for all $\alpha'\in\Sigma_\theta$. Also ${\alpha'}^\vee+{^s\alpha'}^\vee\in X(T^-)$ for all $\alpha'\in\Sigma_\theta$.
\end{Lem}
\begin{proof}
For the first assertion, if $\alpha$ is of type 1 or 2, then $\alpha'={^s\alpha'}$ and hence the assertion is immediate. If $\alpha$ is of type 3 and the diagram of $\alpha$ is $A_1\times A_1$, then again the assertion is clear. Assume $\alpha$ is of type 3 and the diagram is $D_n$. Then with the notation of Section \ref{SS:Type_3_root}, we may write
\[
\alpha'=\alpha_1=e_1-e_n\qand {^s\alpha'}=\alpha_2=e_1+e_n,
\]
where $\theta e_n=e_n$. Hence by restricting to the $\theta$-split torus $T^-$, we have $e_n|_{T^-}=0$, so
\[
\alpha'|_{T^-}=e_1|_{T^-}={^s\alpha'}|_{T^-}.
\]
The first assertion of the lemma is proven.

The second assertion holds because ${\alpha'}^\vee+{^s\alpha'}^\vee$ is either $2(\alpha-\theta\alpha)^\vee$ or $(\alpha-\theta\alpha)^\vee$, and in either case the image is in the $\theta$-split torus $T^-$.
\end{proof}

Therefore, we can fold the root datum $\Phi(G, T)_\theta$ using the involution $s$, giving rise to a subgroup
\[
G^{-}\subseteq G^{--}\subseteq G
\]
whose based root datum is
\[
\overline{\Phi}(G, T^-)^{red}:=(X(T^-), \overline{\Sigma}_\theta, X(T^-)^\vee, \overline{\Sigma}_\theta^\vee),
\]
where each simple root in $\overline{\Sigma}_\theta$ is of the form
\[
\oalpha=\begin{cases}(\alpha-\theta\alpha)|_{T^-}&\text{if $\alpha$ is of type 1};\\
\alpha|_{T^{-}}=\frac{1}{2}(\alpha-\theta\alpha)|_{T^{-}}&\text{if $\alpha$ is of type 2 or 3},\end{cases}
\]
where all the roots are viewed as restricted to $T^-$ and hence the notation $\oalpha$. Each coroot in $\overline{\Sigma}_\theta^\vee$ is of the form
\[
\oalpha^\vee=\begin{cases}(\alpha-\theta\alpha)^\vee&\text{if $\alpha$ is of type 1};
\\\alpha^\vee-\theta\alpha^\vee&\text{if $\alpha$ is of type 2 or 3}.\end{cases}
\]

\begin{Prop}
The root datum $\overline{\Phi}(G, T^-)^{red}$ is the reduced root datum obtained by discarding the short roots from all the irreducible components of type $BC_n$ in $\overline{\Phi}(G, T^-)$. (Recall that $\overline{\Phi}(G, T^-)$ is the $\theta$-split root datum as defined in \eqref{E:theta-split-root-datum}.)
\end{Prop}
\begin{proof}
This follows from the above description of the roots $\oalpha$ of $\overline{\Phi}(G, T^-)^{red}$ and Lemma \ref{L:reduced_split_root_datum}.
\end{proof}

The group $G^-$ can be more explicitly described as follows. The folding $s:\Sigma_\theta\to\Sigma_\theta$ on the simple roots naturally extends to an involution $s:\overline{\Sigma_\theta^{ac}}\to \overline{\Sigma_\theta^{ac}}$ on all the roots. For each $\alpha\in\overline{\Sigma_\theta^{ac}}$, we let $\mathcal{O}_\alpha\subseteq \overline{\Sigma_\theta^{ac}}$ be the orbit of $\alpha$ under $s$. We then let
\[
u_{\oalpha}:\Ga\longrightarrow G,\quad x\mapsto \prod_{\beta\in\mathcal{O}_\alpha}u_\beta(x).
\]
Let $U_\oalpha$ be the image of $u_{\oalpha}$. Then the group $G^-$ is generated by the torus $T^-$ and the ``root subgroups" $U_\oalpha$ for all $\oalpha\in\overline{\Sigma_\theta^{ac}}$. (See \cite[Proposition 10.3.5]{Springer}.)

In particular, if $\alpha\in\Delta\smallsetminus\Delta_0$ then we have
\[
u_{\oalpha}(x)=\begin{cases}u_{\alpha-\theta\alpha}(x)&\text{if $\alpha$ is of type 1};\\
u_{\frac{1}{2}(\alpha-\theta\alpha)}(x)&\text{if $\alpha$ is of type 2};\\
u_{\alpha_1}(x)u_{\alpha_2}(x)&\text{if $\alpha$ is of type 3}.\end{cases}
\]
Note for the third case, one can see that the two elements $u_{\alpha_1}(x)$ and $u_{\alpha_2}(x)$ commute, so $u_{\oalpha}$ is indeed a homomorphism.

Let us mention that the group $G^{--}$ is used to apply the method of folding to construct $G^-$ and plays only an auxiliary role just like $G^{++}$.


\section{Commutativity with principal $\SL_2$}


We have thus far constructed the two subgroups $G^+\subseteq G$ and $G^-\subseteq G$ along with the principal $\SL_2$ homomorphism $\varphi^+:\SL_2\to G^+$ with the assumption that if a simple root $\alpha\in\Delta\smallsetminus\Delta_0$ is of type 3 then we exclude the diagram $\dynkin[edge-length=.5cm, root-radius=.05cm]A{o*}$. In this section, we prove that the image $\varphi^+(\SL_2)$ commutes with $G^-$ pointwise, though later we will have to exclude three other diagrams.

We set
\[
U_{2\rho_0}=\{\varphi^+(\begin{pmatrix}1&x\\&1\end{pmatrix})\st x\in k\}
\qand
U_{-2\rho_0}=\{\varphi^+(\begin{pmatrix}1&\\x&1\end{pmatrix})\st x\in k\}.
\]
Also recall that $T^0$ is the image of $2\rho_0^\vee: \Gm\to G^+$. Then $\varphi^+(\SL_2)$ is generated by $T^0$ and $U_{\pm 2\rho_0}$, and we will show that $G^-$ commutes with these groups pointwise.

\subsection{Commutativity with $T^0$}
Let us first prove the following without any further assumption.
\begin{Lem}\label{L:commutativity_with_T^0}
The two groups $G^{-}$ and $T^0$ commute pointwise.
\end{Lem}
\begin{proof}
We can apply the folding argument that is used to construct $G^-$ to the larger torus $T^0T^-$ to construct a larger group. Indeed, for each type 3 simple root $\alpha$, we know from Proposition \ref{P:decomposition_type_3_root} that $\la\alpha_i,2\rho_0^\vee\ra=0$ for $i=1, 2$, which implies $\alpha_1|_{T^0}=\alpha_2|_{T^0}=0$. Hence one can see that for each $\alpha'\in\Sigma_\theta$, we have the identity $\alpha'|_{T^0T^-}={^s\alpha'}|_{T^0T^-}$ on the restriction to the larger torus $T^0T^-$. This allows us to fold the root datum $\Phi(G, T)_\theta$ with respect to the torus $T^0T^-$ instead of $T^-$, which gives a group whose based root datum is
\[
(X(T^0T^-), \overline{\Sigma}_\theta, X(T^0T^-)^\vee, \overline{\Sigma}_\theta^\vee).
\]
Let us denote this group by $G^{0-}$. The only difference from $G^-$ is that the maximum torus of $G^{0-}$ is the larger torus $T^0T^-$, and we have $G^-\subseteq G^{0-}$.

Now, the action of the involution $\theta$ descends to an action of this root datum. Since each root $\oalpha\in\overline{\Sigma}_\theta$ is of the form $(\alpha-\theta\alpha)|_{T^{0-}}$ or $\frac{1}{2}(\alpha-\theta\alpha)|_{T^{0-}}$ for a simple root $\alpha\in\Delta\smallsetminus\Delta_0$, we know that for every (not necessarily simple) root $\oalpha$, we have $\theta\oalpha=-\oalpha$. Hence one can see that $\oalpha|_{T^0}=0$, meaning that $T^0$ is in the kernel of all the roots of $G^{0-}$. Thus $T^0$ commutes with $G^-$ pointwise.
\end{proof}

\subsection{Other diagrams to be excluded}
To proceed further, we exclude three other diagrams in addition to the diagram \eqref{E:excluded_diagram_type3}.

If $\alpha\in\Delta\smallsetminus\Delta_0$ is of type 1, we exclude the following two
\begin{equation}\label{E:excluded_diagram_type1}
\dynkin C{*o*.**}\qquad\qquad \dynkin F{***o}
\end{equation}
which are, respectively, No.\ 9 and No.\ 16 in Table \ref{T:table} of Appendix \ref{A:table}. If $\alpha\in\Delta\smallsetminus\Delta_0$ is of type 2, we exclude
\begin{equation}\label{E:excluded_diagram_type2}
\dynkin B{o*.**}
\end{equation}
which is No.\ 6 in the table. By Proposition \ref{P:types_of_simple_roots}, to exclude this diagram is equivalent to assuming that
\[
\text{$\alpha$ is of type 2}\quad \Longleftrightarrow\quad \alpha=-\theta\alpha.
\]

Let us explain what goes wrong with these three diagrams. For the two diagrams of type 1, one can see that the simple root $\alpha$ is not strongly orthogonal to some of the $\theta$-stable simple roots in $\Delta_0$.

For the diagram of type 2, as an example, consider the following case
\[
\dynkin[labels={\alpha, \beta}]B{o*}
\]
where $\alpha$ is a long root and $\beta$ is a short root. One can then see that $-\theta$ is the simple reflection associated with the short root $\alpha+\beta$, so that $-\theta\alpha=2\alpha+\beta$ . Then
\[
\frac{1}{2}(\alpha-\theta\alpha)=\alpha+\beta,
\]
which is a short root.  One can then see (up to isogeny) that $G^+=\SL_2$ is generated by the root subgroups $U_{\pm\beta}$ and the principal $\SL_2$ homomorphism is the identity. One can also see that $G^-=\SL_2$ is generated by the root subgroups $U_{\pm(\alpha+\beta)}$. Since $\beta+(\alpha+\beta)=\alpha+2\beta$, which is a root, by the commutator relation \eqref{E:commutator_relation} these two root subgroups do not commute. Indeed, if one assumes that $G=\Sp_4$, one can visualize $G^+$ and $G^-$ by matrices as
\[
G^+=\{\begin{pmatrix}g&\\ &{^tg^{-1}}\end{pmatrix}\st g\in\SL_2\}, \qand
G^-=\{\begin{pmatrix}aI&bJ\\cJ&dI\end{pmatrix}\st \begin{pmatrix}a&b\\c&d\end{pmatrix}\in\SL_2\},
\]
where $I=\left(\begin{smallmatrix}1&0\\0&1\end{smallmatrix}\right)$ and $J=\left(\begin{smallmatrix}0&1\\1&0\end{smallmatrix}\right)$. One can then directly verify that these two subgroups do not commute, though $G^-$ and $T^0$ do commute.

\subsection{Commutativity with root subgroups}
To show that $\varphi^+(\SL_2)$ and $G^-$ commute pointwise (by excluding the four diagrams), it suffices to show that for all the simple roots $\oalpha\in\overline{\Sigma}_\theta$ the root subgroups $U_{\pm\oalpha}$ commute with the unipotent groups $U_{\pm2\rho_0}$. But, for most cases, we will show that the subgroups $U_{\pm\oalpha}$ already commute with the unipotent groups $U_{\pm\gamma}$ for $\gamma\in\Delta_0$.

\begin{Lem}
Let $\alpha\in\Delta\smallsetminus\Delta_0$ be of type 1 and let $\alpha'=\alpha-\theta\alpha$. We exclude the two diagrams \eqref{E:excluded_diagram_type1}. Then for each $\gamma\in\Delta_0$, the additive closure of the set $\{\alpha',\gamma\}$ is $\{\pm\alpha', \pm\gamma\}$; namely $\alpha'$ and $\gamma$ are strongly orthogonal. Therefore, $U_{\pm\oalpha}$ commute with $U_{\pm\gamma}$ for all $\gamma\in\Delta_0$, and hence commute with $U_{\pm2\rho_0}$.
\end{Lem}
\begin{proof}
Since $\alpha'$ and $\gamma$ are orthogonal by Lemma \ref{L:invariant_non-invariant_roots_orthogonal}, if the additive closure of $\{\alpha',\gamma\}$ is not $\{\pm\alpha', \pm\gamma\}$, the only possible scenario is that the additive closure is a root datum of type $B_2$ with both $\alpha'$ and $\gamma$ short roots. Hence the lemma holds if the diagram of $\alpha$ is simply laced or of type $G_2$. Then one can prove the lemma by directly checking all the diagrams of the $B_n$, $C_n$ and $F_4$ types provided in Appendix \ref{A:table}. Thus the commutativity follows from the commutator relation \eqref{E:commutator_relation}.
\end{proof}

\begin{Lem}
Let $\alpha\in\Delta\smallsetminus\Delta_0$ be of type 2. We exclude the diagram \eqref{E:excluded_diagram_type2}. Then for each $\gamma\in\Delta_0$, the additive closure of the set $\{\alpha,\gamma\}$ is $\{\pm\alpha, \pm\gamma\}$; namely $\alpha$ and $\gamma$ are strongly orthogonal. Therefore, $U_{\pm\oalpha}$ commute with $U_{\pm\gamma}$ for all $\gamma\in\Delta_0$, and hence commute with $U_{\pm2\rho_0}$.
\end{Lem}
\begin{proof}
If a $\Z$-linear combination $i\alpha+j\gamma$ with $i>0$ is a root, then $j\geq 0$ and further $\theta(i\alpha+j\gamma)=-i\alpha+j\gamma$ must be a negative root, which implies $j=0$. Similarly, for the other cases. Hence $\alpha$ and $\gamma$ are strongly orthogonal. Thus the commutativity follows from the commutator relation \eqref{E:commutator_relation}.
\end{proof}


\begin{Lem}
Assume that $\alpha\in\Delta\smallsetminus\Delta_0$ is of type 3, whose diagram is
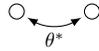
\[
\dynkin[involutions={1<below>[\theta^*]2}, edge/.style={white}]A{oo}
\]
Then both $U_{\pm\alpha}$ and $U_{\mp\theta\alpha}$ commute with $U_{\pm\rho_0}$ pointwise.
\end{Lem}
\begin{proof}
Apparently all the simple roots in $\Delta_0$ are strongly orthogonal to both $\alpha$ and $\theta^*\alpha$. Thus the commutativity follows from the commutator relation \eqref{E:commutator_relation}.
\end{proof}

Now, let us take care of the type 3 case whose diagram is $D_n$. This is the only case where $\alpha$ is not strongly orthogonal to the $\theta$-invariant roots.

\begin{Lem}
Assume that $\alpha\in\Delta\smallsetminus\Delta_0$ is of type 3, whose diagram is
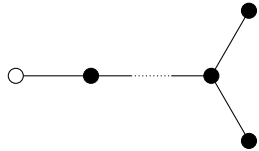
\[
\dynkin[labels={,,,,}]D{o*.***}
\]
so that we can uniquely write $\alpha-\theta\alpha=\alpha_1+\alpha_2$, where $\alpha_1$ and $\alpha_2$ are positive roots orthogonal to $2\rho_0$. Recall then that the root subgroup $U_{\oalpha}$ associated with $\oalpha=\frac{1}{2}(\alpha-\theta\alpha)|_{T^-}$ is generated by the elements of the form $u_{\oalpha}(x)=u_{\alpha_1}(x)u_{\alpha_2}(x)$ for $x\in k$. Then $U_{\pm\oalpha}$ commutes with $U_{\pm2\rho_0}$ pointwise.
\end{Lem}
\begin{proof}
Recall that we use the standard notation for the roots of type $D_n$, so that
\[
\alpha=e_1-e_2\qand \alpha-\theta\alpha=2e_1=(e_1-e_n)+(e_1+e_n)=\alpha_1+\alpha_2,
\]
where we have put
\[
\alpha_1=e_1-e_n\qand\alpha_2=e_1+e_n.
\]
Note that
\[
-\theta e_1=e_1\qand -\theta e_i=-e_i\quad (2\leq i\leq n).
\]

Apparently, the simple roots that do not appear in the diagram of $\alpha$ are strongly orthogonal to $\alpha_1$ and $\alpha_2$, and their root subgroups clearly commute with $U_{\pm\oalpha}$.

Among the simple roots in the diagram of $\alpha$, those that are not strongly orthogonal to $\alpha_1$ or $\alpha_2$ are
\[
e_{n-1}-e_n\qand e_{n+1}+e_n.
\]
It suffices to show that the elements $u_{\alpha_1}(x)u_{\alpha_2}(x)$ and $u_{\pm(e_{n-1}-e_n)}(y)u_{\pm(e_{n-1}+e_n)}(y)$ commute for all $x, y\in k$.

Let $\beta_1=e_{n-1}-e_n$ and $\beta_2=e_{n-1}+e_n$. Consider the additive closure $\{\alpha_1, \alpha_2, \beta_1, \beta_2\}^{ac}$. One can see that this additive closure is a root datum of type $A_3$, where we can choose $\{\beta_1, e_1-e_{n-1}, \beta_2\}$ as a $\theta$-basis, giving rise to the diagram
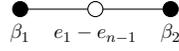
\[
\dynkin[labels={\beta_1,e_1-e_{n-1}, \beta_2}] A{*o*}
\]
Here it should be noted that
\[
(e_1-e_{n-1})-\theta(e_1-e_{n-1})=2e_1=(e_1-e_n)+(e_1+e_n).
\]
Up to isogeny, the root subgroup $U_{\oalpha}$ is realized as
\[
\{\begin{pmatrix}I_2&xI_2\\ &I_2\end{pmatrix}\st x\in k\}
\]
where $I_2$ is the $2\times 2$ identity matrix, and the group generated by $u_{\pm(e_{n-1}-e_n)}(y)u_{\pm(e_{n-1}+e_n)}(y)$ is
\[
\{\begin{pmatrix}g&\\&g\end{pmatrix}\st g\in\SL_2\}.
\]
One can easily verify that these two subgroups commute pointwise.

Similarly one can show the commutativity for $U_{-\oalpha}$. The lemma is proven.
\end{proof}

Hence, we have proven the following.
\begin{Prop}\label{P:commutativity_SL2_general_form}
Assume that no $\alpha\in\Delta\smallsetminus\Delta_0$ has the diagram among the four diagrams in \eqref{E:excluded_diagram_type1}, \eqref{E:excluded_diagram_type2} and \eqref{E:excluded_diagram_type3}. Then the group $G^-$ pointwise commutes with the image $\varphi^+(\SL_2)$ of the principal $\SL_2$ homomorphism $\varphi^+:\SL_2\to G^+$.

Hence we have the group homomorphism
\[
G^-\times\SL_2\longrightarrow G.
\]
\end{Prop}


\section{Dual groups}


Now, we assume $G$ is a connected reductive group split over our nonarchimedean local field $F$ and $\theta$ an $F$-involution on $G$. We let $H$ be the $\theta$-fixed points of $G$, so that $X:=H\backslash G$ is a symmetric variety. Let $T\subseteq G$ be a maximal split torus and set
\[
T^+=\{t\in T\st\theta(t)=t\}^\circ\qand T^-=\{t\in T\st\theta(t)=t^{-1}\}^\circ.
\]
We assume that $T$ is chosen so that $T^-$ is a maximal $\theta$-split torus.

We let
\[
\Phi(G, T)=(X(T),\Phi, X(T)^\vee, \Phi^\vee)
\]
be the root datum of $G$. The involution $\theta$ naturally induces an involution on $\Phi(G, T)$, which we also denote by $\theta$. By our choice of the torus $T$, this involution is admissible. We let $\theta^\vee$ be the involution on the dual datum
\begin{equation}\label{E:dual_datum}
\Phi(\widehat{G}, \widehat{T})=(X(T)^\vee,\Phi^\vee, X(T), \Phi),
\end{equation}
which we have called the dual involution. The dual involution $\theta^\vee$ does not have to be admissible. For example, the involutions
\[
\dynkin B{o**.**}\qand \dynkin C{o**.**}
\]
are dual to each other, yet the first one is admissible but the second one is not. (See No.\ 6 and 8 in Table \ref{T:table}.)

\subsection{Definition of dual group}
Let $\widehat{G}$ be the Langlands dual group of $G$, so that it is the complex connected reductive group whose root datum is dual to that of $G$, namely $\Phi(\widehat{G}, \widehat{T})$ as in \eqref{E:dual_datum}. By applying the constructions in the previous sections, we have the subgroups
\[
\widehat{G}^+\qand \widehat{G}^-
\]
of $\widehat{G}$. We set
\[
\widehat{G}_X=\widehat{G}^-\qand \widehat{\varphi}^+: \SL_2(\C)\longrightarrow \widehat{G}^+,
\]
where $\widehat{\varphi}^+$ is the principal $\SL_2(\C)$ homomorphism for $\widehat{G}^+$.

\begin{Lem}
Let $\alpha^\vee\in\Delta^\vee\smallsetminus\Delta_0^\vee$ be a simple coroot of $G$, which is a simple root of $\widehat{G}$. The diagram of $\alpha^\vee$ (with respect to the dual involution $\theta^\vee$) is not among the four diagrams in \eqref{E:excluded_diagram_type1}, \eqref{E:excluded_diagram_type2} and \eqref{E:excluded_diagram_type3}.
\end{Lem}
\begin{proof}
By Lemma \ref{L:additive_closure_dual}, if $\Phi_{\alpha^\vee}$ is the additive closure of $\{\alpha^\vee\}\cup\Delta_0^\vee$, then its dual $(\Phi_{\alpha^\vee})^\vee$ is the additive closure $\Phi_\alpha$ of $\{\alpha\}\cup\Delta_0$. Hence the diagram of $\alpha^\vee$ is dual to that of $\alpha$. Since $\theta$ is admissible, the diagram of $\alpha$ is also admissible. (See \cite[Proposition 4.9, p.38]{Helminck_commuting_pair}.) But none of the these four diagrams has an admissible dual.
\end{proof}

By this lemma, we can apply Proposition \ref{P:commutativity_SL2_general_form} and obtain a homomorphism
\[
\widehat{\varphi}_X:\widehat{G}_X\times\SL_2(\C)\longrightarrow \widehat{G}.
\]

\section{Trivial representation}


The trivial representation $\one$ of $G$ is $H$-distinguished for any subgroup $H\subseteq G$, so in particular it is $H$-distinguished for any symmetric variety $X=H\backslash G$. Hence by Conjecture (I) the local Langlands parameter $\varphi_\one$ of $\one$ should factor through $\widehat{\varphi}_X$ for any symmetric variety $X$. In this section, we prove this assertion.

Let us first recall that the local Langlands parameter $\varphi_\one$ of the trivial representation $\one$ is given by
\[
\varphi_\one:WD_F\longrightarrow \C^\times\xrightarrow{\;2\rho\;} \widehat{G},
\]
where the first map is given by the square root of the norm map, namely $w\mapsto |w|^{\frac{1}{2}}$ for $w\in WD_F$ (so it is trivial on the $\SL_2(\C)$-factor of $WD_F$), and the second map is the sum of positive coroots of $\widehat{G}$, which we identify with the sum of positive roots of $G$.

The sum of positive roots
\[
2\rho=\sum_{\alpha\in\Phi^+}\alpha
\]
is decomposed as
\[
2\rho=2\rho_-+2\rho_0,
\]
where
\[
2\rho_-=\sum_{\alpha\in\Phi^+\smallsetminus\Phi_0^{+}}\alpha\qand 2\rho_0=\sum_{\alpha\in\Phi_0^{+}}\alpha.
\]
Here, we assume that the order of roots is a $\theta$-order \eqref{E:positivity}.

\begin{Lem}
The involution $\alpha\mapsto-\theta\alpha$ stabilizes both $\Phi\smallsetminus\Phi_0$ and $\Phi^+\smallsetminus\Phi_0^{+}$ setwise.
\end{Lem}
\begin{proof}
Let $\alpha\in \Phi\smallsetminus\Phi_0$. Then it is written as
\[
\alpha=\sum_{\alpha_i\in\Delta\smallsetminus\Delta_0}n_{\alpha_i}\alpha_i+\sum_{\beta_i\in\Delta_0}m_i\beta_i,
\]
where at least one of $n_{\alpha_i}$ is nonzero. Then
\[
-\theta\alpha=-\sum_{\alpha_i\in\Delta\smallsetminus\Delta_0}n_{\alpha_i}\theta\alpha_i-\sum_{\beta_i\in\Delta_0}m_i\beta_i.
\]
We know from Lemma \ref{L:theta^*} that $-\theta\alpha_i=\theta^*\alpha_i+\gamma_i$ for some $\gamma_i\in\Span_{\Z}(\Delta_0)$. Noting that $\theta^*$ permutes the roots in $\Delta\smallsetminus\Delta_0$, we can write
\[
-\theta\alpha=\sum_{\alpha_i\in\Delta\smallsetminus\Delta_0}n_{\alpha_i}\theta^*\alpha_i+\sum_{\beta_i\in\Delta_0}m'_i\beta_i
=\sum_{\alpha_i\in\Delta\smallsetminus\Delta_0}n_{\theta^*\alpha_i}\alpha_i+\sum_{\beta_i\in\Delta_0}m'_i\beta_i
\]
for some $m'_i$'s. Now, since $n_{\alpha_i}\neq 0$ for some $i$, we have $n_{\theta^*\alpha_j}\neq 0$ for some $j$, which implies $-\theta\alpha\in\Phi\smallsetminus\Phi_0$. Further, if $\alpha\in\Phi^+\smallsetminus\Phi_0^{+}$ then $n_{\alpha_i}\geq 0$ for all $i$, which implies $n_{\theta^*\alpha_i}\geq 0$ for all $i$, so $-\theta\alpha\in\Phi^+\smallsetminus\Phi_0^{+}$.
\end{proof}

Now, we are ready to prove our theorem.
\begin{Thm}\label{T:trivial_parameter}
Let $\varphi_\one:WD_F\to \widehat{G}$ be the local Langlands parameter of the trivial representation. Then for any symmetric variety $X=H\backslash G$, the parameter $\varphi_\one$ factors through $\widehat{\varphi}_X:\widehat{G}_X\times\SL_2(\C)\to \widehat{G}$.

More explicitly, we have
\[
2\rho_-:\C^\times\longrightarrow \widehat{G}^-\qand 2\rho_0:\C^\times\longrightarrow \SL_2(\C)\xrightarrow{\widehat{\varphi}^+}\widehat{T}^0,
\]
which give the factorization
\[
\varphi_\one:WD_F\longrightarrow\C^\times\xrightarrow{\,2\rho_-+2\rho_0\,}\widehat{G}^-\times\SL_2(\C)\xrightarrow{\widehat{\varphi}_X}\widehat{G}.
\]
\end{Thm}
\begin{proof}
By the above lemma, we can write
\[
2\rho_-=\sum_{\alpha\in\Phi^+\smallsetminus\Phi_0^{+}}\alpha=\sum_{\alpha\in\Phi^+\smallsetminus\Phi_0^{+}}(-\theta\alpha),
\]
which gives
\[
2\rho_-=\frac{1}{2}\sum_{\alpha\in\Phi^+\smallsetminus\Phi_0^{+}}(\alpha-\theta\alpha).
\]
(Though we do not need it, it should be noted that the right hand side is actually a $\Z$-linear combination of $(\alpha-\theta\alpha)$'s despite the $\frac{1}{2}$; for, if $-\theta\alpha=\alpha$ then $\alpha-\theta\alpha=2\alpha$, and if $-\theta\alpha\neq\alpha$ then there are two occurrences of $\alpha-\theta\alpha$ in the sum.)
Hence
\[
\theta(2\rho_-)=-2\rho_-,
\]
which implies that the image of $2\rho_-$ is in $\widehat{T}^-$. Hence we have the map
\[
2\rho_-:\C^\times\longrightarrow \widehat{T}^-\subseteq \widehat{G}^-.
\]

As for $2\rho_0$, we already know
\[
2\rho_0:\C^\times\longrightarrow \widehat{T}^0\subseteq \widehat{G}^+,
\]
and by definition of the principal $\SL_2$, we know that $2\rho_0$ factors through $\widehat{\varphi}^+:\SL_2(\C)\to\widehat{G}^+$, so we have the map
\[
2\rho_0:\C^\times\longrightarrow \SL_2(\C)\xrightarrow{\widehat{\varphi}^+}\widehat{T}^0.
\]

The theorem follows.
\end{proof}

\section{Examples}


In this section, we verify our conjectures by numerous examples. We occasionally assume that our $p$-adic field $F$ has odd residual characteristic. Whenever we do this (except in Section \ref{SS:orthogonal_period}), this is because the Casselman type criteria of \cite{KT-Cuspidal, KT-Square, Takeda} are used. Also in this section, by $S_k$ we always mean the unique $k$-dimensional irreducible representation of $\SL_2(\C)$.

\subsection{Group Case}
Let us consider the case known as the group case; namely $G=H\times H$ where the involution $\theta$ switches the two factors of $H$. Then the $\theta$-fixed points are the diagonal $\Delta H$, which is of course isomorphic to $H$. Then $\pi\in\Irr(H\times H)$ is $\Delta H$ distinguished if and only if $\pi=\tau\otimes\check{\tau}$, where $\tau\in\Irr(H)$ and $\check{\tau}$ is its contragredient.

On the dual side, $\widehat{G}=\widehat{H}\times \widehat{H}$, and by the choice of positive roots \eqref{E:positivity} the root data of the two copies of $\widehat{H}$ have the opposite ordering.  Also the $\theta$-split torus $\widehat{T}^-$ is of the form $\{(s, s^{-1})\}\subseteq S\times S$, where $S$ is (any choice of) maximal torus of $\widehat{H}$. Note that no root of $\widehat{H}\times \widehat{H}$ is fixed by $\theta$, which implies $\widehat{G}^+$ is trivial and hence the restriction of $\widehat{\varphi}_X:\widehat{G}_X\times\SL_2(\C)\to \widehat{G}$ to $\SL_2(\C)$ is trivial. One can then see that $\widehat{\varphi}_X$ is given by
\[
\widehat{H}\longrightarrow \widehat{H}\times \widehat{H},\quad h\mapsto (h, c_{\widehat{H}}(h)),
\]
where $c_{\widehat{H}}$ is an involution on $\widehat{H}$ which acts on the torus $S\subseteq \widehat{H}$ as $c_{\widehat{H}}(s)=w_{\widehat{H}}(s^{-1})$, where $w_{\widehat{H}}$ is the longest element of the Weyl group of $\widehat{H}$. This involution is known as the Chevalley involution and it is conjectured by Prasad in \cite[Conjecture 2]{Prasad-MVW} that the Chevalley involution of a local Langlands parameter corresponds to the contragredient.

Hence our Conjecture (I) is equivalent to this conjecture of Prasad. Also in the group case the relative matrix coefficients are the usual matrix coefficients, and hence $\tau\otimes\check{\tau}$ is relatively cuspidal (resp.\ relatively square integrable, resp.\ relatively tempered) if and only if $\tau$ is cuspidal (resp.\ square integrable, resp.\ tempered). Thus the three statements in Conjecture (II) are well-known (conjectural) properties of the local Langlands correspondence.

\subsection{Linear period of $\GL_{2n}$.}
Consider the case $X=(\GL_n\times\GL_n)\backslash \GL_{2n}$, so we consider a $\GL_n\times\GL_n$-period of $\GL_{2n}$, which is often known as a linear period. Then $X$ is a symmetric variety with $\theta=\Int(\begin{pmatrix}&J_n\\ J_n&\end{pmatrix})$, where $J_n$ is the $n\times n$ anti-diagonal matrix. One can readily check that the dual of $\theta$ acts on $\GL_{2n}^\vee=\GL_{2n}(\C)$ in the same way. One can then see that the $\theta$-split torus $\widehat{T}^-$ of $\GL_{2n}(\C)$ is
\[
\widehat{T}^-=\{\diag(t_1,\dots,t_n,t_n^{-1},\dots,t_1^{-1})\st t_i\in\C^\times\},
\]
and there is no $\theta$-invariant root. Then the usual choice $\Delta=\{\alpha_1,\dots,\alpha_n,\alpha_{n+1},\dots\alpha_{2n-1}\}$ gives a $\theta$-order on the root datum of $\GL_{2n}(\C)$, where $\alpha_i=e_i-e_{i+1}$. One can then see that
\[
\oalpha_i=\oalpha_{2n-i}=\overline{e_i-e_{i+1}}\quad(1\leq i\leq n-1),\qand \oalpha_n=2\overline{e_n}.
\]
Hence the $\theta$-split root datum is that of $\Sp_{2n}$, so $\widehat{G}_X=\Sp_{2n}(\C)$. Since there is no $\theta$-invariant root, the restriction of $\widehat{\varphi}_X:\widehat{G}_X\times\SL_2(\C)\to \widehat{G}$ to $\SL_2(\C)$ is trivial, and hence we just write $\widehat{\varphi}_X:\widehat{G}_X\to \widehat{G}$, which is nothing but the inclusion $\Sp_{2n}(\C)\subseteq\GL_{2n}(\C)$; namely we have
\[
\widehat{\varphi}_X:\Sp_{2n}(\C)\xhookrightarrow{\quad}\GL_{2n}(\C).
\]

Now, Conjecture (I) implies that if $\pi\in\Irr(\GL_{2n})$ is $\GL_n\times\GL_n$-distinguished then its $L$-parameter $\varphi_{\pi}$ factors through the inclusion $\Sp_{2n}(\C)\subseteq\GL_{2n}(\C)$, or equivalently $\varphi_\pi$ is of symplectic type.

If $n=1$ then $X=T\backslash \GL_2$, where $T$ is a split-torus, and $\widehat{G}_X=\SL_2(\C)$. This is the well-known case studied by Waldspurger \cite{Waldspurger_Toric_period}, according to which $\pi$ has a $T$-period if and only if $\pi$ has a trivial central character. Hence Conjecture (I) along with its converse holds. But recently in his master's thesis \cite{Matsumoto}, Y.\ Matsumoto has shown that if $\pi$ is a non-twisted Steinberg representation then $\pi$ is relatively cuspidal, so that this example shows that the converse of (II-a) does not hold. 

The case for general $n$ has been worked out by many people, especially by relating the linear period with the Shalika period. (See \cite{Jacquet-Rallis_linear_period, Jiang_Nien_Qin, Matringe_linear_period, Matringe_linear-Shalika, Smith_linear, Smith_linera_IJNT, SV}.) In particular, it is known that a square integrable $\pi\in\Irr(\GL_{2n})$ is $\GL_n\times\GL_n$-distinguished if and only if $\pi$ is of symplectic type. (See \cite[Proposition 6.1]{Matringe_linear-Shalika}.) This proves Conjecture (I) and its converse for square integrable $\pi$. Also when $\pi$ is the generalized Steinberg representation $St_k(\rho)$ in the sense of \cite{Matringe_linear-Shalika}, then Theorem 6.1 of \cite{Matringe_linear-Shalika} says that $\pi$ is $\GL_n\times\GL_n$-distinguished if and only if $\rho$ is of symplectic type for $k$ odd and of orthogonal type for $k$ even. But the $L$-parameter of $St_k(\rho)$ is of the form $\varphi_\rho\otimes S_k$, where $\varphi_\rho$ is the $L$-parameter of $\rho$ and $S_k$ is the $k$-dimensional representation of $\SL_2(\C)$. Since $S_k$ is of orthogonal type for $k$ odd and of symplectic type for $k$ even, we know that $St_k(\rho)$ is always of symplectic type. Thus Conjecture (I) holds for $St_k(\rho)$.

Assume $F$ has odd residual characteristic. In \cite[Theorem 6.3]{Smith_linear}, Smith constructed a large family of relatively discrete series $\GL_n\times\GL_n$-distinguished representations which are not themselves discrete series. Later in \cite[Theorem 3.3]{Smith_linera_IJNT} he showed that they are of symplectic type and their $L$-parameters are not in a proper Levi of $\Sp_{2n}(\C)$, showing one direction of Conjecture (II-b) for this family of representations.

Also when $n=2$ and $F$ has odd residual characteristic, some relatively cuspidal $\GL_2\times\GL_2$-distinguished representations are studied in \cite{KT-Tokyo}. The result there is also consistent with Conjectures (I) and (II-a).

\subsection{$\GL_{n-1}\times\GL_1$-period of $\GL_{n}$}
Consider the case $X=(\GL_{n-1}\times\GL_1)\backslash\GL_{n}$. This is a symmetric variety with $\theta=\Int(\begin{pmatrix}&&1\\&I_{n-2}&\\1&&\end{pmatrix})$, where $I_{n-2}$ is the $(n-2)\times(n-2)$ identity matrix. The dual of $\theta$ acts in the same way on the dual $\GL_n^\vee=\GL_n(\C)$, and the $\theta$-invariant torus $\widehat{T}^+$ and the $\theta$-split torus $\widehat{T}^-$ are respectively
\begin{align*}
\widehat{T}^+&=\{\diag(1,t_2,\dots,t_{n-1},1)\st t_i\in\C^\times\}Z_{\GL_n}(\C);\\
\widehat{T}^-&=\{\diag(t,1,\dots,1,t^{-1})\st t\in\C^\times\},
\end{align*}
where $Z_{\GL_n}$ is the center of $\GL_n$. The usual choice of the simple roots $\Delta=\{\alpha_1,\dots,\alpha_{n-1}\}$ gives a $\theta$-order. Then
\[
\Delta_0=\{\alpha_2,\dots,\alpha_{n-2}\}\qand \Sigma_\theta=\{\alpha_1-\theta\alpha_1\}=\{e_1-e_n\};
\]
namely the diagram of $\theta$ is precisely that of type $A_{n-1}$ as in Example \ref{Ex:A_n}.

The $\theta$-split root datum is that of $\SL_2(\C)$. Furthermore, $\alpha_1-\theta\alpha_1=e_1-e_n$ is strongly orthogonal to all the roots in $\Delta_0$. This implies that the two groups $\widehat{G}^+$ and $\widehat{G}^-$ commute with each other. Indeed, one can see that $\widehat{G}^-\times \widehat{G}^+=\SL_2(\C)\times Z_{\GL_n}(\C)\GL_{n-2}(\C)$ and this embeds into $\GL_{n}(\C)$ as
\[
\widehat{\varphi}_X:\SL_2(\C)\times Z_{\GL_n}(\C)\GL_{n-2}(\C)\longrightarrow\GL_n(\C),\quad (\begin{pmatrix}a&b\\c&d\end{pmatrix}, zg)\mapsto z\begin{pmatrix}a&&b\\&g&\\c&&d\end{pmatrix}.
\]
Accordingly, the map $\widehat{\varphi}_X:\SL_2(\C)\times\SL_2(\C)\to\GL_n(\C)$ is given by pre-composing the above embedding with the principal $\SL_2(\C)\to\GL_{n-2}(\C)$ for the second $\SL_2(\C)$.

It is good to observe that the sum of positive coroots of $\GL_n(\C)$ in fact factors through $\widehat{\varphi}_X$, so that the $L$-parameter of the trivial representation factors through it. To see it, recall that the sum of positive coroots of $\widehat{G}^+=Z_{\GL_n}(\C)\GL_{n-2}(\C)$ factors through the principal $\SL_2(\C)\to\GL_{n-2}(\C)$. The remaining positive coroots of $\GL_n(\C)$ (which we identify with roots of $\GL_n$) are
\[
e_1-e_2,\dots,e_1-e_n, e_2-e_n, e_3-e_n,\dots, e_{n-1}-e_n.
\]
The sum of them is $(n-1)(e_1-e_n)$, so its image (when it is viewed as a cocharacter of $\widehat{\GL_n}=\GL_n(\C)$) is of the form
\[
\begin{pmatrix}t^{n-1}&&\\&I_{n-2}&\\ &&t^{-(n-1)}\end{pmatrix}\quad (t\in\C^\times),
\]
which is indeed in the first $\SL_2(\C)$.

In the literature, instead of $\GL_{n-1}\times\GL_1$-distinguished representations, $\GL_{n-1}$-distinguished representations have been more often studied. But if $\pi\in\Irr(\GL_n)$ has a trivial central character, then $\pi$ is $\GL_{n-1}$-distinguished if and only if it is $\GL_{n-1}\times\GL_1$-distinguished because $Z_{\GL_n}\GL_{n-1}=\GL_{n-1}\times\GL_1$. Hence, the $\GL_{n-1}\times\GL_1$-distinguished representations are subsumed under the $\GL_{n-1}$-distinguished representations.

Now, the $\GL_n$-distinguished representations were first studied by Prasad \cite{Prasad-GL_3} for $n=3$, in which he gave a complete list of distinguished representations and conjectures for larger $n$. Later, Venketasubramanian in \cite[Theorem 1.1]{Venketasubramanian} proved the conjecture (with minor modifications), according to which $\pi\in\Irr(\GL_n)$ is $\GL_{n-1}\times\GL_1$-distinguished if and only if either $\pi=\one$ or $\pi=\Ind_{P_{2,n-2}}^{\GL_n}\rho\otimes \one_{\GL_{n-2}}$ for an infinite dimensional representation of $\GL_2$ with trivial central character, where $P_{2, n-2}$ is the $(2, n-2)$-parabolic. In other words, by using the notation of \cite{Venketasubramanian}, $\pi$ is $\GL_{n-1}\times\GL_1$-distinguished if and only if the $L$-parameter $\Lcal(\pi)$ of $\pi$ contains the $L$-parameter $\Lcal(\one_{n-2})$ of the trivial representation $\one_{n-2}$ of $\GL_{n-2}$ and the quotient $\Lcal(\pi)\slash\Lcal(\one_{n-2})$ is an $L$-parameter of an infinite dimensional representation of $\GL_2$. (Since we are assuming the central character of $\pi$ is trivial, the case where the quotient $\Lcal(\pi)\slash\Lcal(\one_{n-2})$ is the character $\nu^{\pm\frac{n-2}{2}}$ of $\GL_2$ in the notation of \cite{Venketasubramanian} never happens.) Such parameter certainly factors through $\widehat{\varphi}_X$. Hence Conjecture (I) is satisfied. The converse of Conjecture (I) fails when the quotient $\Lcal(\pi)\slash\Lcal(\one_{n-2})$ corresponds to the trivial representation of $\GL_2$. But this is the only case where the $L$-parameter factors through $\widehat{\varphi}_X$ but the corresponding representation is not $\GL_{n-1}\times\GL_1$-distinguished.

As for Conjecture (II), Kato and Takano \cite[Proposition 8.2.3]{KT-Cuspidal} showed that if $\rho$ is a supercuspidal representation of $\GL_2$ with trivial central character then the induced representation $\Ind_{P_{2,n-2}}^{\GL_n}\rho\otimes \one$ is relatively cuspidal. Also if $\rho$ is square integrable then $\Ind_{P_{2,1}}^{\GL_3}\rho\otimes \one$ is relatively square integrable (\cite[5.1]{KT-Square}). One can also prove that if $\rho$ is tempered then $\Ind_{P_{2,1}}^{\GL_3}\rho\otimes \one$ is relatively tempered by using the same computation as \cite[5.1]{KT-Square}. These are consistent with Conjecture (II).

\subsection{Symplectic period of $\GL_{2n}$}\label{SS:symplectic_period}
Let us consider $X=\Sp_{2n}\backslash\GL_{2n}$. Let
\[
J_n=\begin{pmatrix}\begin{smallmatrix}0&1\\-1&0\end{smallmatrix}&&\\&\ddots&\\ &&\begin{smallmatrix}0&1\\-1&0\end{smallmatrix}\end{pmatrix},
\]
and define the involution $\theta$ by $\theta(g)=J_n{^tg^{-1}}J_n^{-1}$. Then the set of $\theta$-fixed points is $\Sp_{2n}$. The dual of $\theta$ acts on the dual group $\GL_{2n}(\C)$ in the same way, and
\begin{align*}
\widehat{T}^+&=\{\diag(t_1,t_1^{-1},t_2,t_2^{-1},\dots,t_n,t_n^{-1})\st t_i\in\C^\times\};\\
\widehat{T}^-&=\{\diag(t_1,t_1,t_2,t_2,\dots,t_n,t_n)\st t_i\in\C^\times\}.
\end{align*}
One can readily see that
\[
\theta(e_i-e_{i+1})=\begin{cases}e_i-e_{i+1}&\text{if $i$ is odd};\\-(e_{i-1}-e_{i+2})&\text{if $i$ is even},\end{cases}
\]
and hence the usual choice of simple roots $\Delta=\{e_1-e_2,\dots,e_{2n-1}-e_{2n}\}$ gives a $\theta$-order. Note that
\[
\Delta_0=\{e_i-e_{i+1}\st \text{$i$ is odd}\}\qand \Delta\smallsetminus\Delta_0=\{e_i-e_{i+1}\st \text{$i$ is even}\}.
\]
(We have chosen the above $J_n$ to define $\Sp_{2n}$ precisely because this choice makes the usual choice of the simple roots give a $\theta$-order.)

Since all the roots in $\Delta_0$ are strongly orthogonal, one can see that
\[
\widehat{G}^+=\SL_2(\C)\times\SL_2(\C)\times\cdots\times\SL_2(\C),
\]
which embeds in $\GL_{2n}(\C)$ diagonally. On the other hand, the $\theta$-split root datum is that of $\GL_n(\C)$, so that $\widehat{G}_X=\widehat{G}^-=\GL_{n}(\C)$, and this $\GL_n(\C)$ embeds into $\GL_{2n}(\C)$ as
\[
\begin{pmatrix}a_{11}&\cdots&a_{1n}\\\vdots&\ddots&\vdots\\a_{n1}&\cdots&a_{nn}\end{pmatrix}\mapsto \begin{pmatrix}a_{11}I_2&\cdots&a_{1n}I_2\\\vdots&\ddots&\vdots\\a_{n1}I_2&\cdots&a_{nn}I_2\end{pmatrix},
\]
where $I_2$ is the $2\times 2$ identity matrix. If $V_n$ is an $n$ dimensional complex vector space and $V_2$ a 2 dimensional complex vector space, then the map $\widehat{\varphi}_X:\GL_n(\C)\times\SL_2(\C)\to\GL_{2n}(\C)=\GL(V_n\otimes V_2)$ is realized as the representation $std_n\otimes S_2$, where $std_n$ is the standard representation of $\GL_n(\C)$ on $V_n$ and $S_2$ is the standard representation of $\SL_2(\C)$ on $V_2$; namely we have
\[
\widehat{\varphi}_X:\GL_{n}(\C)\times\SL_2(\C)\longrightarrow\GL(V_n\otimes V_2),\quad \widehat{\varphi}_X=std_n\otimes S_2.
\]

One can then see that the image of each $L$-parameter $\varphi:WD_F\to \widehat{G}_X\times\SL_2(\C)\to\GL_{2n}(\C)$ is of the form
\[
\begin{pmatrix}\varphi_\rho(w)|w|^\frac{1}{2}&\\&\varphi_\rho(w)|w|^{-\frac{1}{2}}\end{pmatrix},
\]
where $\varphi_\rho:WD_F\to \widehat{G}_X=\GL_n(\C)$ is the first component of $\varphi$.

Hence, if $\pi\in\Irr(\GL_{2n}(\C))$ is $\Sp_{2n}$-distinguished then Conjecture (I) implies that $\pi$ has to be a constituent of
\[
I_\rho:=\Ind_{P_{n,n}}^{\GL_{2n}}\left(\rho|\det|^{\frac{1}{2}}\otimes \rho|\det|^{-\frac{1}{2}}\right)
\]
for some $\rho\in\Irr(\GL_n)$, where $P_{n,n}$ is the $(n,n)$-parabolic. Further, Conjecture (II) implies that (a) if $\rho$ is supercuspidal then $\pi$ is relatively cuspidal, (b) $\rho$ is square integrable if and only if $\pi$ is relatively square integrable, and (c) $\rho$ is tempered if and only if $\pi$ is relatively tempered.

Now, the $\Sp_{2n}$-distinguished representations were first studied by Leumos and Rallis \cite{Heumos-Rallis}, in which they showed that for square integrable $\rho$ the Langlands quotient $\pi_\rho$ of $I_\rho$ is distinguished. Then, assuming the residue characteristic of $F$ is odd, Kato and Takano proved that if $\rho$ is supercuspidal then $\pi_\rho$ is relatively cuspidal (\cite[Proposition 8.3.4]{KT-Cuspidal}), and Smith (\cite[Theorem 6.1]{Smith_symplectic}) proved that if $\rho$ is square integrable then $\pi_\rho$ is relatively square integrable. All these results are consistent with our Conjecture.

\subsection{Orthogonal period of $\GL_n$}\label{SS:orthogonal_period}
Consider the case $X=\OO_n\backslash\GL_n$, where $\OO_n$ is the split orthogonal group. This is a symmetric variety with the involution $\theta$ given by $\theta(g)={^tg^{-1}}$. The dual of $\theta$ acts on the dual group $\GL_n(\C)$ in the same way, and the maximal torus of $\GL_n(\C)$ is $\theta$-split. Further, for each root $\alpha$, we have $\theta\alpha=-\alpha$. From these, one can readily see that $\widehat{G}^+=1$ and $\widehat{G}_X=\widehat{G}^-=\GL_n(\C)$. Hence $\widehat{\varphi}_X:\widehat{G}_X\times\SL_2(\C)\to\GL_n(\C)$ is the same as the identity map $\GL_n(\C)\to\GL_n(\C)$, so Conjecture (I) trivially holds.

An interesting question, however, is the converse of Conjecture (I). The converse certainly fails for the simple reason that if the central character $\omega_{\pi}$ of $\pi\in\Irr(\GL_n)$ is not trivial upon restriction to $\OO_{n}$ (namely if $\omega_{\pi}(-1)=-1$), then $\pi$ cannot be $\OO_n$-distinguished. Yet, very recently, Zou \cite{Zou} has shown that, assuming the residue characteristic of $F$ is odd, a supercuspidal $\pi$ is $\OO_n$-distinguished if and only if $\omega_{\pi}(-1)=1$. Also the author has been communicated by Zou that it has been conjectured that $\pi\in\Irr(\GL_n)$ is $\OO_n$-distinguished if and only if $\pi$ is lifted from the Kazhdan-Patterson cover (\cite{KP}) of $\GL_n$. Presumably, one way to modify our Conjectures is to incorporate the theory of $L$-groups for covering groups.


\appendix
\section{Rank one involutions}\label{A:table}


In this appendix, we summarize the results on rank one involutions, especially by computing $\alpha-\theta\alpha$ for a $\theta$-non-invariant simple root $\alpha$. This proves Proposition \ref{P:types_of_simple_roots}. All the results are summarized in the table below, which also incorporates the information on admissibility. (Also see \cite[\textsc{Table} I, p.39]{Helminck_commuting_pair}.)

{

\pgfkeys{/Dynkin diagram, edge-length=.7cm, root-radius=.07cm}

\begin{table}[h]
\caption{Rank one involutions}\label{T:table}
\centering
\renewcommand{\arraystretch}{1.5}
\begin{tabular}{cclcc}
\hline
No. & $\Phi$ & Diagram & Type & Admissible \\\hline\hline
1 &$A_1\times A_1$  & \dynkin[labels={,}, involutions={1<above>[\theta^*]2}, edge/.style={white}]A{oo} & 3  & $+$ \\
2 & $A_1$ & \dynkin A{o} & 2 & $+$ \\
3 & $A_2$ & \dynkin A{o*} & 3 & $-$\\
4 & $A_3$ & \dynkin A{*o*} & 3 & $+$\\
5 & $A_n$ & \dynkin[involutions={1<above>[\theta^*]4}]A{o*.*o} & 1 & $+$\\\hline
6 & \multirow{2}{*}{$B_n$} & \dynkin B{o**.**} & 2 & $+$ \\
7 &       & \dynkin B{*o*.**} & 1 & $-$ \\\hline
8 & \multirow{2}{*}{$C_n$} & \dynkin C{o**.**} & 1 & $-$ \\
9 &       & \dynkin C{*o*.**} & 1 & $+$ \\\hline
10 & \multirow{2}{*}{$D_n$} & \dynkin D{o**.***} & 3 & $+$ \\
11 &       & \dynkin D{*o*.***} & 1 & $-$ \\\hline
12 & $E_6$ & \dynkin E{*o****} & 1 & $-$ \\\hline
13 & $E_7$ & \dynkin E{o******} & 1 & $-$ \\\hline
14 & $E_8$ & \dynkin E{*******o} & 1 & $-$ \\\hline
15 & \multirow{2}{*}{$F_4$} & \dynkin F{o***} & 1 & $-$ \\
16 & & \dynkin F{***o} & 1 & $+$ \\\hline
17 & \multirow{2}{*}{$G_2$} & \dynkin G{*o} & 1 & $-$ \\
18 & & \dynkin G{o*} & 1 & $-$ \\\hline
\end{tabular}
\end{table}

}

\subsection{Type $A_n$ (No.\ 1 -5)}
There are five cases (No.\ 1 - 5 in the table) and the computation of $\alpha-\theta\alpha$ is already done in Section \ref{SS:examples_Dynkin_diagrams}. See Examples \ref{Ex:A_1}, \ref{Ex:A_2}, \ref{Ex:A_3} and \ref{Ex:A_n}.

\subsection{Type $B_n$ (No.\ 6 -7)}
There are two cases (No.\ 6 - 7 in the table). The roots are $\{\pm e_i\pm e_j\}\cup\{\pm e_i\}$ for $1\leq i, j\leq n$. \\

\noindent {No.6.}
\[
\dynkin B{o**.**}
\]
and
\[
\alpha=e_1-e_2,
\]
so the Weyl group element $w_0$ acts as
\[
w_0 e_1=e_1,\qquad w_0e_i=-e_i\quad (2\leq i\leq n),
\]
and $-\theta=w_0$. Hence
\[
\alpha-\theta\alpha=e_1-e_2+w_0(e_1-e_2)=2e_1,
\]
so
\[
\frac{1}{2}(\alpha-\theta\alpha)\in\Phi.
\]

\noindent {No.7.}
\[
\dynkin B{*o*.**}
\]
and
\[
\alpha=e_2-e_3,
\]
so the Weyl group element $w_0$ acts as
\[
w_0 e_1=e_2,\qquad w_0e_i=-e_i\quad (3\leq i\leq n),
\]
and $-\theta=w_0$. Hence
\[
\alpha-\theta\alpha=e_2-e_3+w_0(e_2-e_3)=e_1+e_2,
\]
so
\[
\alpha-\theta\alpha\in\Phi.
\]

\subsection{Type $C_n$ (No.\ 8 -9)}
There are two cases (No.\ 8 - 9 in the table). The roots are $\{\pm e_i\pm e_j\}\cup\{\pm 2e_i\}$ for $1\leq i, j\leq n$.\\

\noindent {No.8.}
\[
\dynkin C{o**.**}
\]
and
\[
\alpha=e_1-e_2,
\]
so the Weyl group element $w_0$ acts as
\[
w_0 e_1=e_1,\qquad w_0e_i=-e_i\quad (2\leq i\leq n),
\]
and $-\theta=w_0$. Hence
\[
\alpha-\theta\alpha=e_1-e_2+w_0(e_1-e_2)=2e_1,
\]
so
\[
\alpha-\theta\alpha\in\Phi.
\]

\noindent {No.9.}
\[
\dynkin C{*o*.**}
\]
and
\[
\alpha=e_2-e_3,
\]
so the Weyl group element $w_0$ acts as
\[
w_0 e_1=e_2,\qquad w_0e_i=-e_i\quad (3\leq i\leq n),
\]
and $-\theta=w_0$. Hence
\[
\alpha-\theta\alpha=e_2-e_3+w_0(e_2-e_3)=e_1+e_2,
\]
so
\[
\alpha-\theta\alpha\in\Phi.
\]

\subsection{Type $D_n$ (No.\ 10 - 11)}
There are two cases (No.\ 10 - 11 in the table). The roots are $\{\pm e_i\pm e_j\}$ for $1\leq i, j\leq n$.\\

\noindent {No.10.}
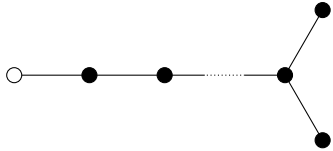
\[
\dynkin D{o**.***}
\]
and
\[
\alpha=e_1-e_2,
\]
so the Weyl group element $w_0$ acts as
\[
\begin{cases}
w_0 e_1=e_1,\quad w_0e_i=-e_i\quad (2\leq i\leq n), &\text{if $n$ is odd};\\
w_0 e_1=e_1,\quad w_0e_i=-e_i\quad (2\leq i\leq n-1),\quad w_0e_n=e_n, &\text{if $n$ is even},
\end{cases}
\]
and
\[
-\theta=\begin{cases}w_0 &\text{if $n$ is odd};\\ \theta^*w_0 &\text{if $n$ is even},\end{cases}
\]
where $\theta^*$ is the nontrivial diagram automorphism. In either case,
\[
\alpha-\theta\alpha=e_1-e_2-\theta(e_1-e_2)=2e_1=(e_1-e_n)+(e_1+e_n),
\]
so $\alpha$ is of type 3.\\

\noindent {No.11.}
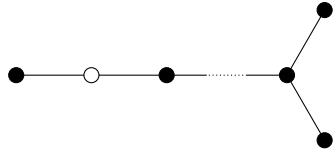
\[
\dynkin D{*o*.***}
\]
and
\[
\alpha=e_2-e_3,
\]
so the Weyl group element $w_0$ acts as
\[
\begin{cases}
w_0 e_1=e_2,\quad w_0e_i=-e_i\quad (3\leq i\leq n), &\text{if $n$ is even};\\
w_0 e_1=e_2,\quad w_0e_i=-e_i\quad (3\leq i\leq n-1),\quad w_0e_n=e_n, &\text{if $n$ is odd},
\end{cases}
\]
and
\[
-\theta=\begin{cases}w_0 &\text{if $n$ is even};\\ \theta^*w_0 &\text{if $n$ is odd},\end{cases}
\]
where $\theta^*$ is the nontrivial diagram automorphism. In either case,
\[
\alpha-\theta\alpha=e_2-e_3-\theta(e_2-e_3)=(e_2-e_3)+(e_1+e_3)=e_1+e_2,
\]
so
\[
\alpha-\theta\alpha\in\Phi.
\]

\subsection{Type $E_6$ (No.\ 12)}
There is one case (No.\ 12 in the table). The roots are
\[
\{\pm(e_i-e_j)\st 1\leq i\neq j\leq 6\}\cup \{\pm (e_7-e_8)\}\cup\{\frac{1}{2}(\pm e_1\pm\cdots\pm e_8)\},
\]
where in the third set there must be exactly four minus signs and $e_7$ and $e_8$ must have opposite signs.

Let us first note that this choice of roots is different from the convection adapted in the book by Knapp \cite{Knapp}, and it is an elementary exercise that the above choice also gives a root system of $E_6$. This convention allows us to label the simple roots as
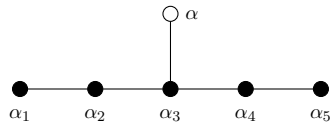
\[
\dynkin[labels={\alpha_1,\alpha,\alpha_2, \alpha_3, \alpha_4, \alpha_5}]E{*o****}
\]
where
\[
\alpha=\frac{1}{2}(e_1+e_2+e_3-e_4-e_5-e_6-e_7+e_8)
\]
and
\[
\alpha_1=e_1-e_2,\quad\alpha_2=e_2-e_3,\quad\cdots,\quad \alpha_5=e_5-e_6.
\]
The Weyl group element $w_0$ acts as
\[
w_0e_1=e_6,\quad w_0e_2=e_5,\quad w_0e_3=e_4,\quad w_0e_7=e_7,\quad w_0e_8=e_8,
\]
and $-\theta=w_0$. Hence
\[
\alpha-\theta\alpha=-e_7+e_8,
\]
so
\[
\alpha-\theta\alpha\in\Phi.
\]

\subsection{Type $E_7$ (No.\ 13)}
There is one case (No.\ 13 in the table). The roots are
\[
\{\pm e_i\pm e_j \st 1\leq i\neq j\leq 6\}\cup\{\pm(e_7-e_8)\}\cup\{\frac{1}{2}(\pm e_1\pm\cdots\pm e_8)\},
\]
where in the third set there must be an even number of minus signs and $e_7$ and $e_8$ must have the opposite signs. The simple roots are
\[
\alpha_1=\frac{1}{2}(e_8-e_7-e_6-e_5-e_4-e_3-e_2+e_1)
\]
and
\[
\alpha_2=e_2+e_1,\quad \alpha_3=e_2-e_1,\quad \alpha_4=e_3-e_1,\quad\cdots,\quad\alpha_7=e_6-e_5,
\]
so that
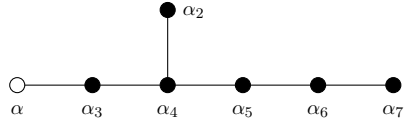
\[
\dynkin[labels={\alpha,\alpha_2,\alpha_3,\alpha_4,\alpha_5,\alpha_6,\alpha_7}]E{o******}
\]
where
\[
\alpha=\alpha_1=\frac{1}{2}(e_8-e_7-e_6-e_5-e_4-e_3-e_2+e_1),
\]
and the Weyl group element $w_0$ is the one for the $D_6$-system, and hence acts as
\[
w_0e_i=-e_i\quad (1\leq i\leq 6),\quad w_0e_7=e_7,\quad w_0e_8=e_8.
\]
Since $-\theta=w_0$,
\[
\alpha-\theta\alpha=\alpha+w_0\alpha=e_8-e_7,
\]
so
\[
\alpha-\theta\alpha\in\Phi.
\]

\subsection{Type $E_8$ (No.\ 14)}
There is one case (No.\ 14 in the table). The roots are
\[
\{\pm e_i\pm e_j\st 1\leq i\neq j\leq 8\}\cup\{\frac{1}{2}(\pm e_1\pm\cdots\pm e_8)\},
\]
where in the second set there must be an even number of minus signs. The simple roots are
\[
\alpha_1=\frac{1}{2}(e_8-e_7-e_6-e_5-e_4-e_3-e_2+e_1)
\]
and
\[
\alpha_2=e_2+e_1,\quad \alpha_3=e_2-e_1,\quad \alpha_4=e_3-e_1,\quad\cdots,\quad\alpha_8=e_7-e_6,
\]
so that
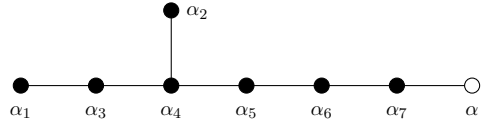
\[
\dynkin[labels={\alpha_1,\alpha_2,\alpha_3,\alpha_4,\alpha_5,\alpha_6,\alpha_7,\alpha}]E{*******o}
\]
where
\[
\alpha=\alpha_8=e_7-e_6.
\]
The Weyl group element $w_0$ acts as
\[
w_0e_i=-e_i\quad (1\leq i\leq 6),\quad w_0e_7=e_8,
\]
and $-\theta=w_0$. Hence
\[
\alpha-\theta\alpha=e_7-e_6+w_0(e_7-e_6)=e_7+e_8
\]
so
\[
\alpha-\theta\alpha\in\Phi.
\]

\subsection{Type $F_4$ (No.\ 15 - 16)}
There are two cases (No.\ 15 - 16 in the table). The roots are
\[
\{\pm e_i\}\cup\{\pm e_i\pm e_j\}\cup\{\frac{1}{2}(\pm e_1\pm e_2\pm e_3\pm e_4)\}
\]
where $1\leq i\neq j\leq 4$, and the simple roots are
\[
\alpha_1=\frac{1}{2}(e_1-e_2-e_3-e_4),\quad \alpha_2=e_4,\quad \alpha_3=e_3-e_4,\quad \alpha_4=e_2-e_3,
\]
so that $\alpha_1$ and $\alpha_2$ are the short roots.\\

\noindent {No.15.}
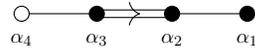
\[
\dynkin[labels={\alpha_4,\alpha_3,\alpha_2,\alpha_1}]F{o***}
\]
so that
\[
\alpha=\alpha_4=e_2-e_3.
\]
The Weyl group element $w_0$ is the longest element in the $C_3$ system, so it acts as
\[
w_0e_1=e_2,\quad w_0e_3=-e_3,\quad w_0e_4=-e_4,
\]
and $-\theta=w_0$. Hence
\[
\alpha-\theta\alpha=(e_2-e_3)+w_0(e_2-e_3)=e_1+e_2,
\]
so
\[
\alpha-\theta\alpha\in\Phi.
\]

\noindent {No.16.}
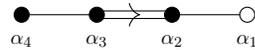
\[
\dynkin[labels={\alpha_4,\alpha_3,\alpha_2,\alpha_1}]F{***o}
\]
so that
\[
\alpha=\alpha_1=\frac{1}{2}(e_1-e_2-e_3-e_4).
\]
The Weyl group element $w_0$ is the longest element in the $B_3$ system, so it acts as
\[
w_0e_1=e_1,\quad w_0e_i=-e_i\quad (2\leq i\leq 4),
\]
and $-\theta=w_0$. Hence
\[
\alpha-\theta\alpha=\alpha+w_0\alpha=\frac{1}{2}(e_1-e_2-e_3-e_4)+\frac{1}{2}(e_1+e_2+e_3+e_4)=e_1
\]
so
\[
\alpha-\theta\alpha\in\Phi.
\]

\subsection{Type $G_2$ (No.\ 17 - 18)}

There are two cases (No.\ 17 - 18 in the table).\\

\noindent{No.\ 17.}
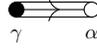
\[
\dynkin[labels={\gamma, \alpha}]G{*o}
\]
where $\gamma$ is a long root and $\alpha$ is a short root, so that the Weyl group element $w_0$ is the simple reflection associated with the long root $\gamma$, and $-\theta=w_0$. Hence
\[
\alpha-\theta\alpha=\alpha+w_0\alpha=2\alpha+\gamma=\text{another short root},
\]
so
\[
\alpha-\theta\alpha\in\Phi.
\]

\noindent{No.\ 18.}
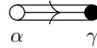
\[
\dynkin[labels={\alpha, \gamma}]G{o*}
\]
where $\gamma$ is a short root and $\alpha$ is a long root, so that the Weyl group element $w_0$ is the simple reflection associated with the short root $\gamma$, and $-\theta=w_0$. Hence
\[
\alpha-\theta\alpha=\alpha+w_0\alpha=2\alpha+3\gamma=\text{another long root},
\]
so
\[
\alpha-\theta\alpha\in\Phi.
\]

\quad\\

\bibliographystyle{amsalpha}
\bibliography{Dual_Group_Symmetric}

\end{document}